\documentclass[11pt,twoside,reqno,centertags,draft]{amsart}
\usepackage{amsfonts}
\usepackage{color,enumitem,graphicx}
\usepackage[colorlinks=true,urlcolor=blue,
citecolor=red,linkcolor=blue,linktocpage,pdfpagelabels,
bookmarksnumbered,bookmarksopen]{hyperref}

  \usepackage{amsmath,amsthm,amsfonts,amssymb}
  \usepackage{ulem}
\begin{document}

\title{ On supercritical nonlinear Schr\"{o}dinger equations with ellipse-shaped potentials}
\date{}
\maketitle

\vspace{ -1\baselineskip}

{\small
\begin{center}
{\sc  Jianfu Yang} \\
Department of Mathematics,
Jiangxi Normal University\\
Nanchang, Jiangxi 330022,
P.~R.~China\\
email: Jianfu Yang: jfyang\_2000@yahoo.com\\[10pt]
{\sc  Jinge Yang*} \\
School of Sciences,
Nanchang Institute of Technology\\
Nanchang 330099,
P.~R.~China\\
email: Jinge Yang: jgyang2007@yeah.net\\[10pt]

\end{center}
}

\renewcommand{\thefootnote}{}

\begin{quote}
{\bf Abstract.}\ \ In this paper, we study the existence and concentration of normalized solutions to the supercritical nonlinear Schr\"{o}dinger equation
\begin{equation*}
\left\{
\begin{array}{l}
-\Delta u + V(x) u = \mu_q u + a|u|^q u \quad {\rm in}\quad \mathbb{R}^2,\\
\int_{\mathbb{R}^2}|u|^2\,dx =1,\\
\end{array}
\right.
\end{equation*}
where $\mu_q$ is the Lagrange multiplier. For ellipse-shaped potentials $V(x)$, we show that for $q>2$ close to $2$, the equation  admits an excited solution $u_q$, and
furthermore, we study the limiting behavior of $u_q$ when $q\to 2_+$. Particularly, we describe precisely the blow-up formation of the excited state $u_q$.
\end{quote}

\newcommand{\N}{\mathbb{N}}
\newcommand{\R}{\mathbb{R}}
\newcommand{\Z}{\mathbb{Z}}

\newcommand{\cA}{{\mathcal A}}
\newcommand{\cB}{{\mathcal B}}
\newcommand{\cC}{{\mathcal C}}
\newcommand{\cD}{{\mathcal D}}
\newcommand{\cE}{{\mathcal E}}
\newcommand{\cF}{{\mathcal F}}
\newcommand{\cG}{{\mathcal G}}
\newcommand{\cH}{{\mathcal H}}
\newcommand{\cI}{{\mathcal I}}
\newcommand{\cJ}{{\mathcal J}}
\newcommand{\cK}{{\mathcal K}}
\newcommand{\cL}{{\mathcal L}}
\newcommand{\cM}{{\mathcal M}}
\newcommand{\cN}{{\mathcal N}}
\newcommand{\cO}{{\mathcal O}}
\newcommand{\cP}{{\mathcal P}}
\newcommand{\cQ}{{\mathcal Q}}
\newcommand{\cR}{{\mathcal R}}
\newcommand{\cS}{{\mathcal S}}
\newcommand{\cT}{{\mathcal T}}
\newcommand{\cU}{{\mathcal U}}
\newcommand{\cV}{{\mathcal V}}
\newcommand{\cW}{{\mathcal W}}
\newcommand{\cX}{{\mathcal X}}
\newcommand{\cY}{{\mathcal Y}}
\newcommand{\cZ}{{\mathcal Z}}

\newcommand{\abs}[1]{\lvert#1\rvert}
\newcommand{\xabs}[1]{\left\lvert#1\right\rvert}
\newcommand{\norm}[1]{\lVert#1\rVert}

\newcommand{\loc}{\mathrm{loc}}
\newcommand{\p}{\partial}
\newcommand{\h}{\hskip 5mm}
\newcommand{\ti}{\widetilde}
\newcommand{\D}{\Delta}
\newcommand{\e}{\epsilon}
\newcommand{\bs}{\backslash}
\newcommand{\ep}{\emptyset}
\newcommand{\su}{\subset}
\newcommand{\ds}{\displaystyle}
\newcommand{\ld}{\lambda}
\newcommand{\vp}{\varphi}
\newcommand{\wpp}{W_0^{1,\ p}(\Omega)}
\newcommand{\ino}{\int_\Omega}
\newcommand{\bo}{\overline{\Omega}}
\newcommand{\ccc}{\cC_0^1(\bo)}
\newcommand{\iii}{\opint_{D_1}D_i}

\theoremstyle{plain}
\newtheorem{Thm}{Theorem}[section]
\newtheorem{Lem}[Thm]{Lemma}
\newtheorem{Def}[Thm]{Definition}
\newtheorem{Cor}[Thm]{Corollary}
\newtheorem{Prop}[Thm]{Proposition}
\newtheorem{Rem}[Thm]{Remark}
\newtheorem{Ex}[Thm]{Example}

\numberwithin{equation}{section}
\newcommand{\meas}{\rm meas}
\newcommand{\ess}{\rm ess} \newcommand{\esssup}{\rm ess\,sup}
\newcommand{\essinf}{\rm ess\,inf} \newcommand{\spann}{\rm span}
\newcommand{\clos}{\rm clos} \newcommand{\opint}{\rm int}
\newcommand{\conv}{\rm conv} \newcommand{\dist}{\rm dist}
\newcommand{\id}{\rm id} \newcommand{\gen}{\rm gen}
\newcommand{\opdiv}{\rm div}

\vskip 0.2cm \arraycolsep1.5pt
\newtheorem{Lemma}{Lemma}[section]
\newtheorem{Theorem}{Theorem}[section]
\newtheorem{Definition}{Definition}[section]
\newtheorem{Proposition}{Proposition}[section]
\newtheorem{Remark}{Remark}[section]
\newtheorem{Corollary}{Corollary}[section]

\section{Introduction}

\bigskip

In this paper, we study the following supercritical Gross-Pitaevskii (GP) type equation
\begin{equation}\label{eq:1.6b}
-\Delta u + V(x) u = \mu_q u + a|u|^q u \quad {\rm in}\quad \mathbb{R}^2,
\end{equation}
that is, $q>2$, with the ellipse-shaped potential
\begin{equation}\label{eq:1.6c}
V(x)=\left( \sqrt{\frac{x_1^2}{b_1^2}+\frac{x_2^2}{b_2^2}}-A\right)^2,
\end{equation}
where $b_1>b_2>0,\ \ A>0$ are constants, and $\mu_q$ is the Lagrange multiplier.

In the case  $q=2$, equation \eqref{eq:1.6b} stems from the study of Bose-Einstein condensation. It was derived independently by Gross and Pitaevskii, and it  is the main theoretical tool for investigating nonuniform dilute Bose gases at low temperatures. Especially, equation \eqref{eq:1.6b} is called the Gross-Pitaevskii equation.

In general, problem \eqref{eq:1.6b} is associated with critical points of the energy functional
\begin{equation}\label{eq:1.3}
E_{a,q}(u)=\frac{1}{2}\int_{\mathbb{R}^2}\big(|\nabla u(x)|^2+V(x)|u(x)|^2\big)\,dx-\frac{a}{q+2}\int_{\mathbb{R}^2}|u|^{q+2}\,dx
\end{equation}
under the constraint
\begin{equation}\label{eq:1.3a}
\int_{\mathbb{R}^2}|u(x)|^2\,dx=1.
\end{equation}
 In particular, in the critical or subcritical case, that is, $q=2$ or $0<q<2$, the functional $E_{a,q}$ is bounded below if $0<a<a^*$ for some $a^*>0$.
Hence, the constrained minimization problem
\begin{equation}\label{eq:1.5}
d_a(q):=\inf_{u\in \mathcal{H}, \ \int_{\mathbb{R}^2}|u|^2\,dx=1}E_{a,q}(u)
\end{equation}
is well defined, where $\mathcal{H}$ is defined by
\[
\mathcal{H}:=\Big\{u\in H^1(\mathbb{R}^2):\int_{\mathbb{R}^2}V(x)|u(x)|^2\,dx<\infty\Big\}.
\]

In the critical case $q = 2$, for $a>0$ the system of Bose-Einstein condensates  collapses if the particle number increases beyond a critical
value; see \cite{DGPS, HMDBB, KMS, SSH} etc. Mathematically, it was proved in \cite{GS} that for a non-negative potential $V(x)$ with finite number zero points, there exists a threshold value $a^*>0$ such that $d_a(2)$ is achieved if $0<a<a^*$,  and
there is no minimizer for $d_a(2)$ if $a\geq a^*$. The threshold value $a^*$ is determined in terms of the solution of the nonlinear scalar field equation
\begin{equation}\label{eq:1.7}
-\Delta u+u=u^3\ \ {\rm in} \ \ \mathbb{R}^2,\ \ u\in H^1(\mathbb{R}^2).
\end{equation}
It is known from \cite{K} that problem \eqref{eq:1.7} admits a unique positive solution up to translations. Such a solution is radially symmetric and exponentially decaying at infinity, see for instance, \cite{BL}. Denote by $Q$ in the sequel the positive solution of \eqref{eq:1.7}. It was found in \cite{GS} that
\begin{equation}\label{eq:1.6}
a^*:=\|Q\|_{L^2(\mathbb{R}^2)}^2.
\end{equation}

The similar symmetry breaking phenomenon was considered in \cite{GZZ} in the subcritical case, i.e. $0<q<2$, for the functional  $E_{a,q}$. When $q$ approaching $2$, the limit behavior  of the minimizer of $E_{a,q}$ constrained by \eqref{eq:1.3a} is described by the unique positive solution $\varphi_q$ of the nonlinear scalar field equation
\begin{equation}\label{eq:1.9}
-\Delta u+\frac{2}{q}u=\frac{2}{q}u^{q+1}, \ \ q>0, \ \ u\in H^1(\mathbb{R}^2)
\end{equation}
in terms of
\begin{equation}\label{eq:1.6a}
a_q^*=\|\varphi_q\|_2^q.
\end{equation}

If the non-negative potential $V(x)$ has infinitely many minima, or particular $V(x)=(|x|-A)^2$, it was shown in \cite{GZZ1} that the
symmetry breaking occurs in the GP minimizers. That is, any
non-negative minimizer of $E_{a,2}$ concentrates at a point on the ring $\{x\in\mathbb{R}^2: |x|=A\}$ as $a\to a^*_-$. Such a result was generalized to the case that the potential $V(x)$ is ellipse-shaped in \cite{GZ}.

The situation becomes different if we turn to the supercritical case $q>2$. Arguments for critical and subcritical cases can not be carried through for the supercritical case, since in the supercritical case the functional $E_{a,q}(u)$ is not bounded below on
the manifold
\[
S(1)=\{u\in \mathcal{H}: \int_{\mathbb{R}^2} |u|^2\,dx = 1\},
\]
and the minimization problem \eqref{eq:1.5} is not well defined.

In this paper, we focus on the supercritical problem \eqref{eq:1.6b} with ellipse-shaped potentials. Although in this case, there is no minimizer for the minimization problem \eqref{eq:1.5}, or no ground state solution for
\eqref{eq:1.6b}, we can find critical points of $E_{a,q}$ constrained on the manifold $S(1)$.
Such a critical point is an excited state solution of \eqref{eq:1.6b}. Actually, for the supercritical case, the functional $E_{a,q}$ has the mountain pass geometry on $S(1)$. This was
revealed in \cite{BV,JE}, and developed to be applied to various problems, see \cite{BS,BJ} etc.  We first look for critical points of $E_{a,q}$ on $S(1)$ by the variant mountain pass theorem, then we investigate the asymptotic behavior of the critical points when $q$ tends to 2.

\bigskip
In the sequel, we denote $|x|_b=\sqrt{b_1^{-2}x_1^2+b_2^{-2}x_2^2}$. Choose $a\in (0,a^*)$ and denote
\begin{equation}\label{eq:1.11}
{\tau_q}=\Big(\frac{2a_q^*}{qa}\Big)^{\frac{1}{q-2}}
\end{equation}
in the sequel. Our existence results are stated as follows.

\begin{Theorem}\label{thm1}
Let $0<a<a^*$. There exists an $\varepsilon_0>0$ such that, for any $q \in (2,2+\varepsilon_0)$, $E_{a, q}(u)$ admits a nonnegative critical point $u_q$ at mountain level on $S(1)$.
\end{Theorem}

One may verify that the functional $E_{a, q}(u)$ has the mountain pass geometry on $S(1)$. Then it is standard to show that $E_{a, q}(u)$ has a Palais-Smale $(PS)$ sequence at the mountain pass level. However, such a $(PS)$ sequence may fail to be bounded. In order to bound $(PS)$ sequence, we use a variant mountain pass theorem inspired of \cite{JE}.

\bigskip

Next, we study the limit behavior of  mountain pass point $u_q$ as $q\rightarrow 2_+$.

\begin{Theorem}\label{thm3} Suppose $0<a<a^*$. Then, for any sequence $\{q_k\}$ with $q_k\to  2_+$ as $k\rightarrow \infty$, there exist $\{x_k\}\subset \mathbb{R}^2$, $\beta>0$ and a subsequence of $\{q_k\}$, still denoted by $\{q_k\}$, such that
\begin{equation}\label{eq:1.13}
\frac{1}{\|\nabla u_{q_k}\|_{L^2(\mathbb{R}^2)}}u_{q_k}\Big(\frac{x}{\|\nabla u_{q_k}\|_{L^2(\mathbb{R}^2)}}+x_k\Big)\rightarrow \frac{\beta}{\|Q\|_{L^2(\mathbb{R}^2)}}Q(\beta x)
\end{equation}
 strongly in  $L^2(\mathbb{R}^2)$.
\end{Theorem}

\bigskip

The proof of Theorem \ref{thm3} is delicate. We commence with the estimate $E_{a,q}(u_q)$ for the mountain pass solution $u_q$. It is shown in section 4 that
\begin{equation*}
\frac{q-2}{2q}\tau_q^2\leq E_{a,q}(u_q)\leq \frac{q-2}{2q}\tau_q^2+\tau_q^{-2}\Big(\frac{1}{2}\int_{\mathbb{R}^2}|x|^2\frac{|Q(x)|^2}{\|Q\|_2^2}\,dx+o(1)\Big)\ \ {\rm as} \ \ q\to 2_+.
\end{equation*}
In contrast with subcritical or critical cases, such an estimate can not established for the supercritical case by simply choosing a suitable trail function. Instead, we need to construct
a suitable path, and estimate energy of $E_{a,q}(u)$ on it. Then the proof of Theorem \ref{thm3} is completed by the blow up analysis. This is different and more difficult than subcritical and critical cases because there is essentially no compactness for the sequence $\{u_q\}$.

\bigskip

Finally, we study the asymptotic behavior of $x_k$.
\begin{Theorem}\label{thm4}
Let $\{x_k\}$ be in Theorem \ref{thm3}. Then either $\liminf_{k\to \infty}|x_k|\rightarrow +\infty$; or
 there exists a subsequence of $\{x_k\}$, still denoted by $\{x_k\}$ such that
 \[
 x_k\rightarrow x_0\in \mathcal{Z}:=\{(b_1A,0), (-b_1A,0)\}
 \]
with
\[
\tau_{q_k}(|x_k|_b-A)\rightarrow 0.
\]
Moreover, in this case equation \eqref{eq:1.13} holds with $\beta=1$,  and $\tau_{q_k}^{-2}\|\nabla u_{q_k}\|_{L^2(\mathbb{R}^2)}^2\rightarrow 1$ as $k\rightarrow \infty$.
\end{Theorem}
\bigskip

This paper is organized as follows. In section 2, we collect and prove some relevant results for future reference. Then, in section 3, we establish the existence of critical points of $E_{a,q}$. Finally, we analyze the limiting behavior of these critical points in section 4.

\bigskip

\section{Preliminaries}

\bigskip

In this section, we collect and prove some relevant results for future reference. The following result is used frequently in literatures, we state it explicitly.


\bigskip

\begin{Lemma}\label{lem:5.1} Let $\varphi_q\geq 0$ be the unique radially symmetric positive solution of \eqref{eq:1.9} with $2<q\leq 3$. Then, $\varphi_q\rightarrow Q$ strongly in $H^1(\mathbb{R}^2)$ as $q\to 2_+$, and there exist positive constants $C$ and $\delta$ independent of $q$ such that
\begin{equation}\label{eq:5.1}
\varphi_q(x)\leq Ce^{-\delta |x|}\ \ {\rm for} \ \ x\in\mathbb{R}^2.
\end{equation}
In particular, $a_q^*=\|\varphi\|_2^q\to a^* = \|Q\|_2^2$ as $q\to2_+$.
\end{Lemma}

\bigskip

 Let $\tilde{S}(1)=\{u\in H^1(\mathbb{R}^2):\|u\|_2^2=1\}$. Denote by $E_{a,q}|_{V=0}(u)$ the energy functional without the trapping potential
\begin{equation}\label{eq:3.5}
E_{a,q}|_{V=0}(u)=\frac{1}{2}\int_{\mathbb{R}^2}|\nabla u|^2\,dx-\frac{a}{q+2}\int_{\mathbb{R}^2}|u|^{q+2}\,dx,\ \ u\in H^1(\mathbb{R}^2).
\end{equation}
We define
\begin{equation}\label{eq:3.6}
\tilde{c}_q=\inf_{g\in \Gamma}\max_{u\in g}E_{a,q}|_{V=0}(u),
\end{equation}
where
\[
\tilde{\Gamma}=\{g\in C([0,1],\tilde{S}(1))|g(0)=u_1, \ g(1)=u_2\}
\]
 for $u_1, u_2\in \tilde{S}(1)$ such that  $\tilde{c}_q>\max\{E_{a,q}|_{V=0}(g(0)),E_{a,q}|_{V=0}(g(1))\}$.
We know from Theorem 2.1 in \cite{JE} that for every $a>0$, there exists a unique positive critical point $\tilde{\varphi}_q$ of $E_{a,q}|_{V=0}(u)$ constrained on $S(1)$ at the energy level $\tilde{c}_q$. Now, we give an explicit description of $\tilde{c}_q$ and $\tilde{\varphi}_q$ in terms of $a, a^*_q$ and $\varphi_q$.
\begin{Lemma}\label{lem:3.1}
Let $q>2$. Then
\begin{equation}\label{eq:3.7}
\tilde{c}_q=\frac{q-2}{2q}\tau_q^2.
\end{equation}
and
\[
\tilde{\varphi}_q(x)=\frac{{\tau_q}}{\|\varphi_q\|_2}\varphi_q({\tau_q}x).
\]
\end{Lemma}
\begin{proof}
It is know from \cite{JE} that
\[
\tilde{c}_q=E_{a,q}|_{V=0}(\tilde{\varphi}_q),\ \, \|\tilde{\varphi}_q\|_2^2=1
\]
and $\tilde{\varphi}_q$ satisfies
\begin{equation}\label{eq:3.8}
-\Delta \tilde{\varphi}_q + \mu_q\tilde{\varphi}_q = a\tilde{\varphi}_q^{q+1},
\end{equation}
where $\mu_q$ is the Lagrange multiplier. Let $w_q=t\tilde{\varphi}_q(sx)$, where $s$ and $t$ satisfies
\[
s^2\mu_q=at^{-q}s^2=\frac{2}{q}.
\]
Then, $w_q$ is a positive solution of \eqref{eq:1.9}. By the uniqueness of positive solutions of \eqref{eq:1.9}, we have $w_q=\varphi_q$. Since $\|\tilde{\varphi}_q\|_2^2=1$,
\begin{equation}\label{eq:3.9}
1 = \int_{\mathbb{R}^2}|\tilde \varphi_q|^2\,dx=t^{-2}s^2\int_{\mathbb{R}^2}|w_q|^2\,dx = t^{-2}s^2 (a_q^*)^{\frac{2}{q}}.
\end{equation}
Solving
\[
at^{-q}s^2=\frac{2}{q}\quad{\rm and }\quad t^{-2}s^2 (a_q^*)^{\frac 2q}=1,
\]
we obtain
\[
t=\Big(\frac{2a_q^*}{qa}\Big)^{\frac{1}{2-q}}(a_q^*)^{\frac{1}{q}},\ \  s=\Big(\frac{2a_q^*}{qa}\Big)^{\frac{1}{2-q}} = \tau_q^{-1}.
\]
Hence, we get
\begin{equation}\label{eq:3.10}
\tilde{\varphi}_q(x)=\frac{\tau_q}{\|\varphi_q\|_2}\varphi_q(\tau_q x).
\end{equation}
Substituting $\tilde{\varphi}_q$ into $E_{a,q}|_{V=0}(u)$,  we obtain
\[
\tilde{c}_q=\frac{q-2}{2q}\Big(\frac{2a_q^*}{qa}\Big)^{\frac{2}{q-2}}.
\]
\end{proof}

\bigskip

\begin{Remark}\label{r1}
It is known from Lemma 2.10 in \cite{JE} that $\tilde{c}_q$ can also be described by
\begin{equation}\label{eq:3.11}
\tilde{c}_q =\inf_{u\in \mathcal{Q}}E_{a,q}|_{V=0}(u),
\end{equation}
 where
\begin{equation*}
\mathcal{Q}=\{u\in H^1(\mathbb{R}^2)|\,\|u\|_2^2=1, \ \tilde{Q}_q(u)=0\}
\end{equation*}
and
\begin{equation}\label{eq:3.12}
\tilde{Q}_q(u)=\int_{\mathbb{R}^2}|\nabla u|^2\,dx-\frac{qa}{q+2}\int_{\mathbb{R}^2}|u|^{q+2}\,dx.
\end{equation}
\end{Remark}

\bigskip

\section{Existence}

\bigskip

In this section, we show the existence of  critical points of the functional $E_{a,q}(u)$ on the sphere
\begin{equation}\label{eq:2.1}
S(1)=\{u\in \mathcal{H}, \|u\|_2^2=1\},
\end{equation}
that is, we will prove Theorem \ref{thm1}. This will be done by a variant mountain pass theorem.
Before we proceed further, we recall the following compactness lemma, which can be proved as that in \cite{YY}.

\begin{Lemma}\label{lem:2.1} Suppose $V\in L_{loc}(\mathbb{R}^2)$ and  $\lim_{x\to\infty}V(x)=\infty$. Then the embedding $\mathcal{H}\hookrightarrow L^q(\mathbb{R}^2)$ is compact for $ q\in[2,\infty)$.
\end{Lemma}

\bigskip

{\bf Proof of Theorem \ref{thm1}.} Obviously, $E_{a, q}(u)\geq E_{a,q}|_{V=0}(u)$ for every $u\in \mathcal{H}$. We claim that
\begin{equation}\label{ee-0a}
\inf_{u\in R_q} E_{a,q}|_{V=0}(u)\geq \frac{q-2}{2q}\tau_q^2,
\end{equation}
where
\begin{equation}\label{1.1-1}
R_q=\{u \in S(1): \|\nabla u\|_2^2=\tau_q^2\},
\end{equation}
which implies
\[
\inf_{u\in R_q} E_{a, q}(u)\geq\inf_{u\in R_q} E_{a,q}|_{V=0}(u)\geq\frac{q-2}{2q}\tau_q^2.
\]
Indeed, by the Gagliardo-Nirenberg inequality (see \cite{W}), we get
\[
E_{a,q}|_{V=0}(u)\geq \frac{1}{2}\int_{\mathbb{R}^2}|\nabla u(x)|^2\,dx-\frac{a}{2a_q^*}\Big(\int_{\mathbb{R}^2}|\nabla u(x)|^2\,dx\Big)^{\frac{q}{2}}\int_{\mathbb{R}^2}| u(x)|^2\,dx.
\]
Hence
\begin{equation}\label{ee-1b}
\inf_{u\in R_q} E_{a,q}|_{V=0}(u)\geq \frac{1}{2}\tau_q^2-\frac{a}{2a_q^*}\tau_q^q=\frac{q-2}{2q}\tau_q^2.
\end{equation}

Let $\varphi\in C_c(\mathbb{R}^2)$ be a nonnegative function such that $\|\varphi\|_2^2=1$, and denote $\varphi^t(x)=t\varphi(tx)$. Then,
\[
E_{a,q}(\varphi^t)=\frac{1}{2}t^2\int_{\mathbb{R}^2}|\nabla \varphi(x)|^2\,dx+\frac{1}{2}\int_{\mathbb{R}^2}V(t^{-1}x)|\varphi(x)|^2\,dx-\frac{at^q}{q+2}\int_{\mathbb{R}^2}
|\varphi(x)|^{q+2}\,dx.
\]
Let $t_1=2^{(q-2)^{-2}}$.   Since
\[
\lim_{q\to 2^+}\int_{\mathbb{R}^2}V(t_1^{-1}x)|\varphi(x)|^2\,dx=V(0)
\]
and
\begin{equation*}
\lim_{q\to 2^+}\frac{\frac{at_1^q}{q+2}\int_{\mathbb{R}^2}|\varphi(x)|^{q+2}\,dx}{\frac{1}{2}t_1^2\int_{\mathbb{R}^2}|\nabla \varphi(x)|^2\,dx}=\lim_{q\to 2^+}\frac{2a}{q+2}\frac{\|\varphi\|_{q+2}^{q+2}t_1^{q-2}}{\|\nabla \varphi\|_2^2}=+\infty,
\end{equation*}
we have
\begin{equation*}
\begin{split}
E(\varphi^{t_1})\leq 2V(0)-\frac{at_1^q}{2(q+2)}\|\varphi\|_{q+2}^{q+2}=2V(0)-\frac{a2^{q(q-2)^{-2}}}{2(q+2)}\|\varphi\|_{q+2}^{q+2}.
\end{split}
\end{equation*}
Hence
\[
\lim_{q\rightarrow 2_+}E_{a,q}(\varphi^{t_1})=-\infty.
\]
It implies that there exists $\varepsilon_0>0$ such that
\begin{equation}\label{1.1-2}
E_{a,q}(\varphi^{t_1})<0
\end{equation}
if $q \in (2, 2+\varepsilon_0)$.
Let
\[
c_q=\inf_{g\in \Gamma_q}\max_{t\in[0,1]}E_{a,q}(g(t)),
\]
where
\[
\Gamma_q=\{g\in C([0,1], S(1)): g(0)=\varphi,\ \ g(1)=\varphi^{t_1}\}.
\]
Since
$
\lim_{q\rightarrow 2_+}\frac{q-2}{2q}\tau_q^2
=+\infty,
$
we deduce
\begin{equation*}
\lim_{q\to2^+}\frac{\|\nabla \varphi^{t_1}\|_2^2}{\tau_q^2}=\lim_{q\to2^+}\frac{t_1^2\|\nabla \varphi\|_2^2}{\tau_q^2}
=\lim_{q\to2^+}\Big(\frac{2^{(q-2)^{-1}}qa}{2a_q^*}\Big)^{2(q-2)^{-1}}\|\nabla \varphi\|_2^2
=+\infty,
\end{equation*}
that is,
\[
\lim_{q\to2^+}\frac{t_1^2}{\tau_q^2}=+\infty.
\]
Hence, by \eqref{ee-0a},\eqref{1.1-2}  for any $g\in \Gamma_q$, if $q>2$ and $q$ is close to $2$, there holds
\begin{equation*}
\begin{split}
\max_{t\in [0,1]} E_{a,q}(g(t))\geq \inf_{u\in R_q}E_{a,q}|_{V=0}(u)
\geq \frac{q-2}{2q}\tau_q^2
&>\max\{E_{a,q}(\varphi), E_{a,q}(\varphi^{t_1}) \}.
\end{split}
\end{equation*}
As a result,
\begin{equation}\label{1.1-3}
c_q\geq \frac{q-2}{2q}\tau_q^2>\max\{E_{a,q}(\varphi), E_{a,q}(\varphi^{t_1})\}.
\end{equation}

Equations \eqref{1.1-2} and \eqref{1.1-3} indicate that the functional $E_{a,q}$ has the mountain pass geometry, and then there exists a $(PS)$ sequence of $E_{a,q}$ at the mountain pass level. However, such a $(PS)$ sequence may fail to be bounded. In order to bound (PS) sequence, we use the following variant mountain pass theorem.

Let $H(u,s)=e^su(e^sx)$. The functional $\mathcal{E}_{a,q}: \mathcal{H}\times \mathbb{R}\rightarrow \mathbb{R}$ is defined as
\begin{equation} \label{eb}
\begin{split}
\mathcal{E}_{a,q}(u,s)&=:E_{a,q}(H(u,s))\\
&=\frac{1}{2}e^{2s}\int_{\mathbb{R}^2}|\nabla u(x)|^2\,dx+\frac{1}{2}
\int_{\mathbb{R}^2}V(e^{-s}x)|u(x)|^2\,dx-\frac{a}{q+2}e^{sq}\int_{\mathbb{R}^2}|u(x)|^{q+2}\,dx.
\end{split}
\end{equation}
Denote the set of paths by
\[
\mathcal{P}_q=\{\gamma\in C([0,1], S(1)\times \mathbb{R}): \gamma(0)=(\varphi, 0), \gamma(1)=(\varphi^{t_1},0)\}
\]
and define
\[
b_q=\inf_{\gamma\in \mathcal{P}_q}\max_{t\in [0,1]}\mathcal{E}_{a,q}(\gamma(t))=\inf_{\gamma\in \mathcal{P}_q}
\max_{t\in [0,1]}E_{a,q}(H(\gamma(t))).
\]
Since $\Gamma_q=\{H(\gamma):\gamma\in \mathcal{P}_q\}$, we get
\begin{equation}\label{b=c}
b_q=c_q.
\end{equation}
By \eqref{1.1-3}, $\mathcal{E}_{a,q}(\gamma(0))=E_{a,q}(\varphi)$ and $\mathcal{E}_{a,q}(\gamma(1))=E_{a,q}(\varphi^{t_1})$, we have
\begin{equation*}
b_q=c_q>\max\{E_{a,q}(\varphi), E_{a,q}(\varphi^{t_1})\}=\max\{\mathcal{E}_{a,q}(\gamma(0)), \mathcal{E}_{a,q}(\gamma(1))\}.
\end{equation*}
Let $Y=\mathcal{H}\times \mathbb{R}$ be the space with the norm
\[
\|(u,s)\|_Y^2=\|(u,s)\|_{\mathcal{H}}^2+\|s\|_\mathbb{R}^2\ \ {\rm for \ \ any}\ \ (u,s)\in Y.
\]
 Then, $S(1)\times \mathbb{R}$ is a submanifold of $Y$ of codimension $1$ and its tangent subspace at a given point $(u,s)\in S(1)\times \mathbb{R}$ is
\[
T_{(u,s)} \big(S(1)\times \mathbb{R}\big)=\{(v_1,s_1)\in Y: \langle v_1,u\rangle_{L^2(\mathbb{R}^2)}=0\}.
\]
Denote by $\mathcal{E}_{a,q}|_{S(1)\times \mathbb{R}}$ the trace of $\mathcal{E}_{a,q}$ on $S(1)\times \mathbb{R}$. Then,  $\mathcal{E}_{a,q}|_{S(1)\times \mathbb{R}}$ is a $C^1$ functional on $S(1)\times \mathbb{R}$, and for any $(u,s)\in S(1)\times \mathbb{R}$,
\[\langle\mathcal{E}_{a,q}'|_{S(1)\times \mathbb{R}}((u,s)), w\rangle=\langle\mathcal{E}_{a,q}'((u,s)), w\rangle\ \
{\rm for} \ \ w\in T_{(u,s)} \big(S(1)\times \mathbb{R}\big).
\]
Moreover,
$$
\mathcal{E}_{a,q}'|_{S(1)\times \mathbb{R}}\in \Big[T_{(u,s)} \big(S(1)\times \mathbb{R}\big)\Big]^{-1}.
$$
By \eqref{b=c}, there exists $\{(g_n, 0)\}\subset\mathcal{P}_q$ such that
\begin{equation}
c_q=\lim_{n\rightarrow \infty}\sup_{g_n}E_{a,q}(g_n)=\lim_{n\rightarrow \infty}\sup_{(g_n, 0)}\mathcal{E}_{a,q}(g_n, 0)=b_q.
\end{equation}
Because $E_{a,q}(|g_n|)\leq E_{a,q}(g_n)$, let $\gamma_n=(|g_n|,0)$, we have
\begin{equation} \label{1.1-5}
\lim_{n\rightarrow \infty}\sup_{\gamma_n}\mathcal{E}_{a,q}(\gamma_n)=b_q=c_q.
\end{equation}
By Theorem 3.2 in \cite{GS}, there exists $\{(w_n, s_n)\}\subset S(1)\times \mathbb{R}$ such that

\begin{equation} \label{1.1-6}
\lim_{n\rightarrow \infty}\mathcal{E}_{a,q}(w_n, s_n)=b_q=c_q,
\end{equation}
\begin{equation} \label{1.1-7}
\lim_{n\rightarrow \infty}\|\mathcal{E}_{a,q}'|_{S(1)\times \mathbb{R}}(w_n, s_n)\|_{\mathcal{Z}_n}=0\ \ with \ \ \mathcal{Z}_n=\big[T_{(w_n,s_n)} \big(S(1)\times \mathbb{R}\big)\big]^{-1},
\end{equation}
and
\begin{equation*}
\lim_{n\rightarrow \infty} dist((w_n, s_n), (|g_n|, 0))=0.
\end{equation*}

Let $u_n=H(w_n, s_n)$. Choose $v_n\in |g_n|$ such that
$$
\|u_n-v_n\|^2_{\mathcal{H}}+|s_n|^2=\Big(dist((w_n, s_n), (|g_n|, 0))\Big)^2.
$$
Then
\begin{equation} \label{1.1-8}
\lim_{n\rightarrow \infty}\|u_n-v_n\|^2_{\mathcal{H}}+|s_n|^2=0.
\end{equation}
By \eqref{1.1-6},
\begin{equation} \label{1.1-9}
\lim_{n\rightarrow \infty}E_{a,q}(u_n)=\lim_{n\rightarrow \infty}\mathcal{E}_{a,q}(w_n,s_n)=c_q.
\end{equation}
For any $(\varphi,\tau)\in T_{(u,s)}\big(S(1)\times \mathbb{R}\big)$ with $(u,s)\in S(1)\times \mathbb{R}$, we claim that
\begin{equation} \label{1.1-10}
\langle \mathcal{E}_{a,q}'|_{S(1)\times \mathbb{R}}(u,s), (\varphi,\tau)\rangle=
\langle E_{a,q}'|_{S(1)}(H(u,s)), H(\varphi,s)\rangle+Q_q(H(u,s))\tau,
\end{equation}
where
\begin{equation} \label{1.1-11}
Q_q(u)=\int_{\mathbb{R}^2}|\nabla u(x)|^2\,dx-\frac{1}{2}\int_{\mathbb{R}^2}x\cdot \nabla V(x)|u(x)|^2\,dx-\frac{qa}{q+2}\int_{\mathbb{R}^2}|u(x)|^{q+2}\,dx.
\end{equation}

Indeed,
\begin{equation*}
\begin{split}
&~~\langle \mathcal{E}'_{a,q}|_{S(1)\times \mathbb{R}}(u,s), (\varphi,\tau)\rangle\\
&=
\langle \mathcal{E}'_{a,q}(u,s), (\varphi,\tau)\rangle\\
&=\lim_{t\rightarrow 0}t^{-1}\big[\mathcal{E}_{a,q}(u+t\varphi,s+t\tau)-\mathcal{E}_{a,q}(u,s)\big]\\
&=\lim_{t\rightarrow0}t^{-1}\big[\mathcal{E}_{a,q}(u+t\varphi,s+t\tau)-\mathcal{E}_{a,q}(u,s+t\tau)+
\mathcal{E}_{a,q}(u,s)-\mathcal{E}_{a,q}(u+t\varphi,s)\big]\\
&~~+\lim_{t\rightarrow 0}t^{-1}\big[\mathcal{E}_{a,q}(u,s+t\tau)-\mathcal{E}_{a,q}(u,s)\big]+
\lim_{t\rightarrow 0}t^{-1}\big[\mathcal{E}_{a,q}(u+t\varphi,s)-\mathcal{E}_{a,q}(u,s)\big]\\
&:=I_1+I_2+I_3.
\end{split}
\end{equation*}
We deduce that
\begin{equation*}
\begin{split}
I_1=&\lim_{t\rightarrow0}t^{-1}\Big\{\frac{1}{2}(e^{2(s+t\tau)}-e^{2s})\int_{\mathbb{R}^2}(|\nabla u+t\nabla \varphi|^2-|\nabla u|^2)\,dx\\
&+\frac{1}{2}\int_{\mathbb{R}^2}\big[\big(\frac{|x|_b}{s+t\tau}-A\big)^2-\big(\frac{|x|_b}{s}-A\big)^2\big]
\big[|u+t\varphi|^2-|u|^2\big]\,dx\\
&-\frac{a}{q+2}\big[e^{q(s+t\tau)}-e^{qs}\big]
\int_{\mathbb{R}^2}\big[|u+t\varphi|^{q+2}-|u|^{q+2}\big]\,dx\Big\}\\
=&\lim_{t\rightarrow0}t^{-1}\Big\{\frac{1}{2}e^{2s}(e^{2t\tau}-1)\int_{\mathbb{R}^2}(2t\nabla u\cdot \nabla \varphi+t^2|\nabla \varphi|^2)\,dx\\
&-\frac{t\tau}{2s(s+t\tau)}\int_{\mathbb{R}^2}|x|_b\big(\frac{|x|_b}{s+t\tau}+\frac{|x|_b}{s}-2A\big)(2tu\varphi+t^2\varphi^2)
\,dx\\
&-\frac{a}{q+2}e^{qs}(e^{qt\tau}-1)\int_{\mathbb{R}^2}\big[|u+t\varphi|^{q+2}-|u|^{q+2}\big]\,dx\Big\}\\
=&0,
\end{split}
\end{equation*}
and
\begin{equation*}
I_2=\lim_{t\rightarrow0}t^{-1}\big[E_{a,q}(e^{s+t\tau}u(e^{s+t\tau}x))-E_{a,q}(e^su(e^sx))\big]=\tau Q_q(e^su(e^sx)).
\end{equation*}
Since $\langle \varphi,u\rangle _{L^2(\mathbb{R}^2)}=0$, we have
\[
e^s\varphi(e^sx)\in T_{e^su(e^sx)}S(1):=\{\psi\in \mathcal{H}:\langle \psi, e^su(e^sx)  \rangle_{L^2(\mathbb{R}^2)}=0\}
\]
and
\begin{equation*}
\begin{split}
I_3&=\lim_{t\rightarrow0}t^{-1}\big[E_{a,q}(e^su(e^sx)+te^s\varphi(e^sx))-E_{a,q}(e^su(e^sx))\big]\\
&=\langle E_{a,q}'(e^su(e^sx)), e^s\varphi(e^sx)\rangle\\
&=\langle E_{a,q}'|_{S(1)}(e^su(e^sx)), e^s\varphi(e^sx)\rangle.
\end{split}
\end{equation*}
As a result, the claim \eqref{1.1-10} holds.

For any $(\varphi,\tau)\in T_{(w_n,\tau_n)}\big(S(1)\times \mathbb{R}\big)$, we infer from \eqref{1.1-7} and \eqref{1.1-10} that,
   \begin{equation} \label{1.1-12}
   \langle E_{a,q}'|_{S(1)}(u_n), H(\varphi, s_n)\rangle+Q_q(u_n)\tau=o(\|(\varphi, \tau)\|_{Y}^2).
\end{equation}
Choosing particularly $(\varphi,\tau)=(0,1)$ in \eqref{1.1-12}, we obtain
\begin{equation} \label{1.1-13}
Q_q(u_n)\rightarrow 0\ \ {\rm as}\ \ n\to \infty.
\end{equation}

On the other hand, for any $\psi\in T_{u_n}S(1)$, we may choose $\varphi$ such that $\psi(x)=e^{s_n}\varphi(e^{s_n}x)$. Therefore,
$$
(\varphi,s_n)\in T_{(w_n, s_n)}\big(S(1)\times \mathbb{R}\big).
$$
Taking $\tau=0$ in \eqref{1.1-12}, we have
\[
\langle E'_{a,q}|_{S(1)}(u_n), \psi\rangle=o(\|\varphi\|_{\mathcal{H}}^2).
\]
By \eqref{1.1-8},
\begin{equation} \label{1.1-14aa}
|s_n|\rightarrow 0
\end{equation}
as $n\rightarrow \infty$ and $\|\psi\|_{\mathcal{H}}\leq 2\|\varphi\|_{\mathcal{H}}$ . Consequently,
$$
\langle E'_{a,q}|_{S(1)}(u_n), \psi\rangle=o(\|\psi\|_{\mathcal{H}}^2),
$$
that is
\begin{equation} \label{1.1-14}
\|E_{a,q}'|_{S(1)}(u_n)\|_{\mathcal{H}^{-1}}\rightarrow 0\ \ as \ \ n\rightarrow \infty.
\end{equation}
 The identity
\begin{equation} \label{1.1-14a}
\frac{q-2}{2}\int_{\mathbb{R}^2}|\nabla u_n(x)|^2\,dx+\frac{1}{2}\int_{\mathbb{R}^2}
\big[q(|x|_b-A)^2+2|x|_b(|x|_b-A)\big]|u_n(x)|^2\,dx=qE_{a,q}(u_n)-Q_q(u_n)
\end{equation}
and $\|u_n\|_2^2=1$ yield
\begin{equation*}
\begin{split}
&qE_{a,q}(u_n)-Q_q(u_n)\\
&\geq \frac{q-2}{2}\int_{\mathbb{R}^2}|\nabla u_n(x)|^2\,dx+\frac{1}{2}\int_{\mathbb{R}^2}
q(|x|_b-A)^2|u_n(x)|^2\,dx+\frac{1}{2}\int_{\mathbb{R}^2}|x|_b^2|u_n(x)|^2\,dx-\frac{1}{2}A^2.
\end{split}
\end{equation*}
This together with \eqref{1.1-9} and \eqref{1.1-13} implies that $\{u_n\}$ is bounded in $\mathcal{H}$. By Lemma \ref{lem:2.1}, we may assume that
\begin{equation} \label{1.1-15a}
u_n\rightharpoonup u_q\in \mathcal{H}\ \ {\rm weakly\ \ in}\ \ \mathcal{H};
\end{equation}
and
\begin{equation} \label{1.1-15b}
u_n\to u_q \ \ {\rm strongly\ \ in}  \ \ L^\gamma(\mathbb{R}^2)\ \ {\rm for\ \ any}\ \ \gamma\geq 2.
\end{equation}
We know from \eqref{1.1-12} and Lemma 3 in [5] that  there exists a $\mu_q^n$ such that
\begin{equation} \label{1.1-15}
-\Delta u_n+Vu_n-a|u_n|^qu_n-\mu_q^nu_n\to 0\ \ {\rm in} \ \ \mathcal{H}^{-1},
\end{equation}
and then,
\begin{equation} \label{1.1-16}
\mu_q^n=\int_{\mathbb{R}^2}|\nabla u_n(x)|^2\,dx+\int_{\mathbb{R}^2}V(x)|u_n(x)|^2\,dx-a\int_{\mathbb{R}^2}|u_n(x)|^{q+2}\,dx+o(\|u_n\|_{\mathcal{H}}),
\end{equation}
which is bounded. Without loss of generality, we assume that $\mu_q^n\to \mu_q$ as $n\to \infty$. Passing the limit $n\to\infty$ in \eqref{1.1-15}, we obtain
\begin{equation} \label{1.1-17}
-\Delta u_q+Vu_q-a|u_q|^qu_q-\mu_qu_q=0.
\end{equation}
By \eqref{1.1-16} and \eqref{1.1-17}, we have
\begin{equation*}
\begin{split}
&\lim_{n\to\infty}\bigg(\int_{\mathbb{R}^2}|\nabla u_n(x)|^2\,dx+\int_{\mathbb{R}^2}V(x)|u_n(x)|^2\,dx\bigg)\\
&=a\int_{\mathbb{R}^2}|u_q|^{q+2}\,dx+\mu_q\\
&=\int_{\mathbb{R}^2}|\nabla u_q(x)|^2+\int_{\mathbb{R}^2}V(x)|u_q(x)|^2\,dx,
\end{split}
\end{equation*}
 namely,
 $$
 \lim_{n\to \infty}\|u_n\|_{\mathcal{H}}=\|u_q\|_{\mathcal{H}}.
 $$
By \eqref{1.1-8} and \eqref{1.1-14aa}, we obtain $v_n\to u_q$  in $\mathcal{H}$. Note that $v_n$ is positive, we get $u_q\geq 0$ and $u_q$ is a solution of \eqref{1.1-17}.  Then  the strong maximum principle implies $u_q>0$.

Since $u_n\to u_q$  in $\mathcal{H}$, by \eqref{1.1-9} and \eqref{1.1-13},
$$
E_{a,q}(u_q)=c_q
$$
and
\begin{equation}\label{1.1-18}
Q_q(u_q)=\int_{\mathbb{R}^2}|\nabla u_q(x)|^2\,dx-\int_{\mathbb{R}^2}|x|_b(|x|_b-A)|u_q(x)|^2\,dx
-\frac{qa}{q+2}\int_{\mathbb{R}^2}|u_q(x)|^{q+2}\,dx=0.
\end{equation}
The proof is complete.\quad$\Box$

\bigskip

\section{Energy estimates and asymptotic behavior}

\bigskip

In this section, we first establish an asymptotic expansion of the energy $E_{a,q}(u_q)$ at the mountain pass point $u_q$. This relies on, among other things, a carefully
choosing a path. Next, we study asymptotic behavior of $u_q$ as $q\to2_+$ and prove Theorem \ref{thm3} and Theorem \ref{thm4}.
\begin{Proposition}\label{ee} There holds
\begin{equation*}
\frac{q-2}{2q}\tau_q^2\leq E_{a,q}(u_q)\leq \frac{q-2}{2q}\tau_q^2+\tau_q^{-2}\Big(\frac{1}{2}\int_{\mathbb{R}^2}|x|^2\frac{|Q(x)|^2}{\|Q\|_2^2}\,dx+o(1)\Big)\ \ {\rm as} \ \ q\to 2_+.
\end{equation*}
\end{Proposition}
\begin{proof}
By \eqref{1.1-3} and $c_q=E_{a,q}(u_q)$, it suffices to prove that
\[
E_{a,q}(u_q)
\leq \frac{q-2}{2q}\tau_q^2+\tau_q^{-2}\Big(\frac{1}{2}\int_{\mathbb{R}^2}|x|^2\frac{|Q(x)|^2}{\|Q\|_2^2}\,dx+o(1)\Big)\ \ {\rm as} \ \ q\to 2_+.
\]
To this purpose, we will construct a path $g\in \Gamma_q$ linking $\varphi$ and $\varphi^{t_1}$ so that
\[
E_{a,q}(u_q)\leq \max_{t\in[0,1]}E_{a,q}(g(t))\leq \frac{q-2}{2q}\tau_q^2+
\tau_q^{-2}\Big(\frac{1}{2}\int_{\mathbb{R}^2}|x|^2\frac{|Q(x)|^2}{\|Q\|_2^2}\,dx+o(1)\Big)
.
\]
The path $g$ is constructed in three parts.

First, we construct a path $g_1$ linking $\varphi$ to some $w_q^{\tilde{t}_0}$ and estimate $E_{a,q}(g_1(s))$.
Let $\varphi_q$ be the positive solution of \eqref{eq:1.9}.  Denote
$$
w_q(x)=\frac{\tau_q}{\|\varphi_q\|_2}\varphi_q(\tau_q (x-x_0)
$$
with
$x_0\in \mathcal{Z}$. By Lemma 8.12 in \cite{CA}, $\varphi_q$ satisfies
\begin{equation} \label{ee-2}
\int_{\mathbb{R}^2}|\nabla \varphi_q(x)|^2\,dx=\int_{\mathbb{R}^2}|\varphi_q(x)|^2\,dx
=\frac{2}{q+2}\int_{\mathbb{R}^2}|\varphi_q(x)|^{q+2}\,dx.
\end{equation}
Note that $a_q^*=\|\varphi_q\|_2^q$, we obtain
\begin{equation} \label{ee-3}
\int_{\mathbb{R}^2}|\nabla w_q(x)|^2\,dx=\tau_q^2
\end{equation}
and
\begin{equation} \label{ee-4}
\int_{\mathbb{R}^2}|w_q(x)|^{q+2}\,dx=\frac{q+2}{2a_q^*}\tau_q^q.
\end{equation}
Choosing $\tilde{t}_0=\sqrt{\frac{q-2}{12q}}$ such that $w_q^{\tilde{t}_0}(x) = \tilde{t}_0w_q(\tilde{t}_0x)$ satisfies
\begin{equation} \label{ee-5a}
\int_{\mathbb{R}^2}|\nabla w_q^{\tilde{t}_0}|^2\,dx=\frac{q-2}{12q}\tau_q^2,
\end{equation}
we define
\[
g_1(s):=\frac{sw_q^{\tilde{t}_0}+(1-s)\varphi}{\|sw_q^{\tilde{t}_0}+(1-s)\varphi\|_2},\ \ s\in[0,1],
\]
where $\varphi\in C_c(\mathbb{R}^2)$ is a nonnegative function with $\|\varphi\|_2^2=1$ given in $\Gamma_q$.

By \eqref{ee-5a} and the fact that $\|w_q^{\tilde{t}_0}\|_2^2=\|\varphi\|_2^2=1$, we have
\begin{equation}\label{ee-5}
\begin{split}
\frac{1}{2}\|\nabla g_1(s)\|_2^2
&\leq 2(s^2\|\nabla w_q^{\tilde{t}_0}\|_2^2+(1-s)^2\|\nabla \varphi\|_2^2)\\
&\leq 2(\|\nabla w_q^{\tilde{t}_0}\|_2^2+\|\nabla \varphi\|_2^2)\\
&\leq \frac{(q-2)}{4q}\tau_q^2,
\end{split}
\end{equation}
since $\lim_{q\to 2^+}\frac{q-2}{q}\tau_q^2=+\infty$. On the other hand,
\begin{equation}\label{ee-6a}
\begin{split}
\int_{\mathbb{R}^2}V(x)|g_1(x)|^2\,dx\leq 4\big(\int_{\mathbb{R}^2}V(x)|w_q^{\tilde{t}_0}|^2\,dx+\int_{\mathbb{R}^2}V(x)|\varphi|^2\,dx\big).
\end{split}
\end{equation}
By Lemma \ref{lem:5.1},
$$\varphi_q(x)\leq Ce^{-\delta |x|},$$
where $C>0,\ \ \delta>0$ independent of $q$,   and Lebesgue dominated theorem,
\begin{equation}\label{ee-6}
\begin{split}
\int_{\mathbb{R}^2}V(x)|w_q^{\tilde{t}_0}(x)|^2\,dx
&=\int_{\mathbb{R}^2}\big(\tilde{t}_0^{-1}|\tau_q^{-1}y+x_0|_b-A\big)^2\frac{|\varphi_q(y)|^2}{\|\varphi_q\|_2^2}\,dy\\
&\leq 2\big(\tilde{t}_0^{-2}\|\varphi_q\|_2^{-2}\int_{\mathbb{R}^2}|\tau_q^{-1}y+x_0|_b^2|\varphi_q(y)|^2\,dy+A^2\big)\\
&=2\big(\tilde{t}_0^{-2} (\|Q\|_2^{-2}\int_{\mathbb{R}^2}|x_0|_b^2|Q(y)|^2\,dy + o(1)) +A^2\big] \\
&=2\big[\tilde{t}_0^{-2}(A^2+o(1))+A^2\big]\\
&\leq 4A^2\tilde{t}_0^{-2}\ \ {\rm as} \ \ q\to 2_+.
\end{split}
\end{equation}
Since $\tilde{t}_0^{2}=\frac{q-2}{12q}$, \eqref{ee-5} and \eqref{ee-6} yield
\begin{equation}\label{ee-7}
\begin{split}
E_{a,q}&\leq \frac{3(q-2)}{12q}\tau_q^2+\frac{48qA^2}{q-2}\\
&\leq \frac{(q-2)}{3q}\tau_q^2
\end{split}
\end{equation}
for $q>2$ and $q$ close to $2$.

Next, we construct the second part $g_2$ of the path $g$. Choose $\tilde{t}_1=\tau_q^{-1}t_1:=\tau_q^{-1}2^{(q-2)^{-2}}$. Let
\begin{equation*}
g_2(s)=\frac{sw_q^{\tilde{t}_1}+(1-s)\varphi^{t_1}}{\|sw_q^{\tilde{t}_1}+(1-s)\varphi^{t_1}\|_2},\ \ s\in[0,1].
\end{equation*}
By \eqref{ee-3}, \eqref{ee-4} and $\|w_q^{\tilde{t}_1}\|_2^2=\|\varphi^{t_1}\|_2^2=1$, we have
\begin{equation} \label{ee-8}
\begin{split}
\frac{1}{2}\|\nabla g_2(s)\|_2^2
&\leq 2s^2\|\nabla w_q^{\tilde{t}_1}\|_2^2+2(1-s)^2\|\nabla \varphi^{t_1}\|_2^2\\
&=2^{\frac{2}{(q-2)^2}+1}s^2+\|\nabla \varphi\|_2^2 2^{\frac{2}{(q-2)^2}+1}(1-s)^2
\end{split}
\end{equation}
and
\begin{equation}\label{ee-9}
\begin{split}
\frac{a}{q+2}\int_{\mathbb{R}^2}|g_2(s)|^{q+2}\,dx
&\geq \frac{4^{-q-2}a}{q+2}\int_{\mathbb{R}^2}|sw_q^{\tilde{t}_1}+(1-s)\varphi^{t_1}|^{q+2}\,dx\\
&\geq \frac{4^{-q-3}a}{q+2}\big[\int_{\mathbb{R}^2}|sw_q^{\tilde{t}_1}|^{q+2}\,dx
+\int_{\mathbb{R}^2}|(1-s)\varphi^{t_1}|^{q+2}\,dx\big]\\
&=\frac{4^{-q-3}a}{q+2}\big[\frac{q+2}{2a_q^*}s^{q+2}2^{\frac{q}{(q-2)^2}}+(1-s)^{q+2}2^{\frac{q}{(q-2)^2}}
\int_{\mathbb{R}^2}|\varphi|^{q+2}\,dx\big].
\end{split}
\end{equation}
If $0\leq s\leq \frac{1}{2}$, then
\begin{equation}\label{ee-10}
\begin{split}
\frac{1}{2}\|\nabla g_2(s)\|_2^2-\frac{a}{q+2}\int_{\mathbb{R}^2}|g_2(s)|^{q+2}\,dx
&\leq C_12^{\frac{2}{(q-2)^2}}-C_2\frac{4^{-q-3}a}{q+2}2^{\frac{q}{(q-2)^2}}\\
&\leq -C_2\frac{4^{-q-3}a}{2(q+2)}2^{\frac{q}{(q-2)^2}}.
\end{split}
\end{equation}
Similarly, the inequality \eqref{ee-10} holds for $\frac{1}{2}\leq s\leq 1$. We may show as \eqref{ee-6} that
\begin{equation} \label{ee-11}
\int_{\mathbb{R}^2}V(x)|g_2(s)|^2\,dx\leq C.
\end{equation}
Therefore, we deduce from from \eqref{ee-10} and \eqref{ee-11} that
\begin{equation} \label{ee-11a}
E_{a,q}(g_2(s))<0
\end{equation}
for $q>2$ and close to $2$.

Finally, we construct a path linking $w_q^{\tilde{t}_0}$ and $w_q^{\tilde{t}_1}$. Since
$$\|\nabla w_q^{\tilde{t}_0}\|_2^2=\frac{q-2}{12q}\tau_q^2\ \ {\rm and}\ \  \|\nabla w_q^{\tilde{t}_1}\|_2^2=2^{\frac{2}{(q-2)^2}},
$$
there exists a $\tilde{t}_2\in (\tilde{t}_0, \tilde{t}_1)$ such that $\|\nabla w_q^{\tilde{t}_2}\|_2^2=\tau_q^2$. By \eqref{ee-0a}, we have
\begin{equation} \label{ee-12}
E_{a,q}(w_q^{\tilde{t}_2})\geq \frac{q-2}{2q}\tau_q^2.
\end{equation}
Let
\begin{equation*}
g_3(s)=w_q^{s\tilde{t}_0+(1-s)\tilde{t}_1}, \ \ s\in[0,1].
\end{equation*}

 Now, we define a path $g$ by $g_1$, $g_2$ and $g_3$, that is,
 \begin{equation*}
g(s)=\left\{
\begin{array}{ll}
\displaystyle
g_1(3s),\ \ &if \ \ 0\leq s\leq \frac{1}{3};\\[6pt]
g_3(3s-1), \ \ &if \ \ \frac{1}{3}\leq s\leq \frac{2}{3};\\[6pt]
g_2(3s-2), \ \ &if \ \ \frac{2}{3}\leq s\leq 1.
\end{array}
\right.
\end{equation*}
Obviously, $g\in\Gamma_q$.

It follows from \eqref{ee-7}, \eqref{ee-11a} and \eqref{ee-12} that
\begin{equation}\label{ee-12a}
\begin{split}
\max_{0\leq s\leq 1}E_{a,q}(g(s))
&=\max_{\frac{1}{3}\leq s\leq \frac{2}{3}}E_{a,q}(g_3(3s-1))\\
&=\max_{\tilde{t}_0\leq t\leq \tilde{t}_1}E_{a,q}(w_q^t).
\end{split}
\end{equation}
Let $f(t)=E_{a,q}(w_q^t)$. It implies by \eqref{ee-3} and \eqref{ee-4} that
\[
f(t)=\frac{1}{2}\tau_q^2t^2-\frac{1}{q}\tau_q^2t^q+\frac{1}{2}\int_{\mathbb{R}^2}V(t^{-1}x)|w_q(x)|^2\,dx.
\]
By \eqref{ee-7}, \eqref{ee-11a} and \eqref{ee-12}, there exists $t_{q}\in (\tilde{t}_0, \tilde{t}_1)$ such that
\begin{equation} \label{ee-13}
f(t_{q})=\max_{t\in(\tilde{t}_0, \tilde{t}_1)}f(t)
\end{equation}
and
\begin{equation} \label{ee-14}
f'(t_{q})=\tau_q^2t_{q}-\tau_q^2t_{q}^{q-1}-\frac{1}{t_{q}^3}\int_{\mathbb{R}^2}|x|_b^2|w_q(x)|^2\,dx
+\frac{A}{t_{q}^2}\int_{\mathbb{R}^2}|x|_b|w_q(x)|^2\,dx.
\end{equation}
By Lemma 2.1 and the Lebegue dominated theorem, we have
\begin{equation}\label{ee-15}
\begin{split}
\int_{\mathbb{R}^2}|x|_b^2|w_q(x)|^2\,dx
&=\int_{\mathbb{R}^2}|\tau_q^{-1}x+x_0|_b^2\|\varphi_q\|_2^{-2}|\varphi_q(x)|^2\,dx\\
&\to \|Q\|_2^{-2}\int_{\mathbb{R}^2}|x_0|_b^2|Q(x)|^2\,dx
\end{split}
\end{equation}
and
\begin{equation} \label{ee-16}
\int_{\mathbb{R}^2}|x|_b|w_q(x)|^2\,dx\to \|Q\|_2^{-2}\int_{\mathbb{R}^2}|x_0|_b|Q(x)|^2\,dx
\end{equation}
as $q\to 2_+$. Since $t_{q}>\tilde{t}_0= \sqrt{\frac{q-2}{12q}}$, we see from \eqref{ee-14}--\eqref{ee-16} that
\begin{equation} \label{ee-17}
(q-2)^{\frac{3}{2}}\tau_q^2|1-t_{q}^{q-2}|\leq C.
\end{equation}
Now, we claim that
\begin{equation} \label{ee-18}
|1-t_{q}|\leq \tau_q^{-\frac{3}{2}}.
\end{equation}
Indeed, if this is not the case, there would exist $q_n\to 2_+$ such that either
\begin{equation*}
t_{1q_n}\geq 1+\tau_{q_n}^{-\frac{3}{2}} \ \ {\rm or} \ \ t_{1q_n}\leq 1-\tau_{q_n}^{-\frac{3}{2}}.
\end{equation*}
If $t_{1q_n}\geq 1+\tau_{q_n}^{-\frac{3}{2}}$, then
\begin{equation*}
\begin{split}
(q_n-2)^{\frac{3}{2}}\tau_{q_n}^2|1-t_{q_n}^{q_n-2}|
&\geq (q_n-2)^{\frac{3}{2}}\tau_{q_n}^2\big[(1+\tau_{q_n}^{-\frac{3}{2}})^{q_n-2}-1\big]\\
&=(q_n-2)^{\frac{3}{2}}\tau_{q_n}^2\big[e^{(q_n-2)\log(1+\tau_{q_n}^{-\frac{3}{2}})}-1\big]\\
&\geq (q_n-2)^{\frac{3}{2}}\tau_{q_n}^2\frac{1}{2}(q_n-2)\tau_{q_n}^{-\frac{3}{2}}\to +\infty,
\end{split}
\end{equation*}
which contradicts \eqref{ee-17}. The case $t_{1q_n}\leq 1-\tau_{q_n}^{-\frac{3}{2}}$ can be ruled out in the same way.

Next, we deal with $$\int_{\mathbb{R}^2}V(x)|w_q^{t_{q}}(x)|^2\,dx.$$ It holds
\begin{equation}\label{ee-19aa}
\begin{split}
\int_{\mathbb{R}^2}V(x)|w_q^{t_{q}}(x)|^2\,dx
&=\int_{\mathbb{R}^2}
\big(|\frac{x}{t_{q}\tau_q}+\frac{x_0}{t_{q}}|_b-|x_0|_b\big)^2\frac{|\varphi_q(x)|^2}{\|\varphi_q\|_2^2}\,dx\\
&=\frac{1}{\tau_q^2}\int_{\mathbb{R}^2}\big(|\frac{x}{t_{q}}+\frac{x_0\tau_q}{t_{q}}|_b
-|\frac{x_0\tau_q}{t_{q}}|_b
+|\frac{x_0\tau_q}{t_{q}}|_b-|x_0\tau_q|_b\big)^2\frac{|\varphi_q(x)|^2}{\|\varphi_q\|_2^2}\,dx.
\end{split}
\end{equation}
By \eqref{ee-18} and $|x_0|_b=A$, we get
\begin{equation}\label{ee-19}
\begin{split}
\bigg|\frac{x}{t_{q}}+\frac{x_0\tau_q}{t_{q}}\bigg|_b-\bigg|\frac{x_0\tau_q}{t_{q}}\bigg|_b
&=\frac{1}{t_{q}}\Big(|x+x_0\tau_q|_b-|x_0\tau_q|_b\Big)\\
&=\frac{1}{t_{q}}\frac{{|x|_b^2}+2(b_1^{-2}x_1, b_2^{-2}x_2)\cdot x_0\tau_q}{|x+x_0\tau_q|_b+|x_0\tau_q|_b}\\
&\to \frac{(b_1^{-2}x_1, b_2^{-2}x_2)\cdot x_0}{A}
\end{split}
\end{equation}
and
\begin{equation}\label{ee-20}
\bigg|\big|\frac{x_0\tau_q}{t_{q}}\big|_b-\big|x_0\tau_q\big|\bigg|_b\leq |x_0\tau_q|_b\frac{|1-t_{q}|}{|t_{q}|}\leq C\tau_q^{-\frac{1}{2}}\to 0
\end{equation}
as $q\to 2$. Since $Q$ is radially symmetric about the origin, by Lemma 2.1, \eqref{ee-19}, \eqref{ee-20} and the Lebesgue dominated theorem, we have
\begin{equation} \label{ee-21}
\begin{split}
\int_{\mathbb{R}^2}V(x)|w_q^{t_{q}}(x)|^2\,dx
&=\tau_q^{-2}\Big(\int_{\mathbb{R}^2}\frac{|b_1^{-2}x_1x_0^1|^2}{A^2}
\frac{|Q(x)|^2}{\|Q\|_2^2}\,dx+o(1)\Big)\\
&=\tau_q^{-2}\Big(\frac{1}{2b_1^2}\int_{\mathbb{R}^2}
\frac{|x|^2|Q(x)|^2}{\|Q\|_2^2}\,dx+o(1)\Big).
\end{split}
\end{equation}
We derive from \eqref{ee-3}, \eqref{ee-4}, \eqref{ee-20} and \eqref{ee-21} that
\begin{equation*}
\begin{split}
\max_{\tilde{t}_0\leq t\leq \tilde{t}_1}E_{a,q}(w_q^t)&=E_{a,q}(w_q^{t_{q}})\\
&=\frac{1}{2}\tau_q^2t_{q}^2-\frac{1}{q}\tau_q^2t_{q}^q+\int_{\mathbb{R}^2}V(x)|w_q^{t_{q}}(x)|^2\,dx\\
&\leq \tau_q^2\max_{t>0}\{\frac{1}{2}t^2-\frac{1}{q}t^q\}+\int_{\mathbb{R}^2}V(x)|w_q^{t_{q}}(x)|^2\,dx\\
&\leq \frac{q-2}{2q}\tau_q^2+\tau_q^{-2}\Big(\frac{1}{2b_1^2}\int_{\mathbb{R}^2}
\frac{|x|^2|Q(x)|^2}{\|Q\|_2^2}\,dx+o(1)\Big),
\end{split}
\end{equation*}
which and  \eqref{ee-12a} immediately yield the conclusion.
\end{proof}

\bigskip

\begin{Proposition}\label{tg1}
There holds
\begin{equation} \label{tg1-1}
C_1(q-2)\tau_q^2\leq \int_{\mathbb{R}^2}|\nabla u_q(x)|^2\,dx\leq \tau_q^2+\frac{C_2}{q-2}
\end{equation}
with $C_1>0$, $C_2>0$.
\end{Proposition}
\begin{proof}
We deduce from Proposition \ref{ee}, \eqref{1.1-14a} and  \eqref{1.1-18} that
\begin{equation} \label{tg1-2}
(q-2)\int_{\mathbb{R}^2}|\nabla u_q(x)|^2\,dx+\int_{\mathbb{R}^2}[q(|x|_b-A)^2+2|x|_b(|x|_b-A)]|u_q(x)|^2\,dx
=(q-2)\tau_q^2+o(1),
\end{equation}
which implies
\[
\int_{\mathbb{R}^2}|\nabla u_q(x)|^2\,dx\leq \tau_q^2+\frac{C}{q-2}
\]
since $\|u_q\|_2^2=1$.

Now, we derive the lower bound in \eqref{tg1-1}.

Suppose on the contrary that there exists $q_k\to 2_+$ such that
\[
\int_{\mathbb{R}^2}|\nabla u_{q_k}(x)|^2\,dx=o((q_k-2)\tau_{q_k}^2) \ \ {\rm as} \ \ k\to \infty.
\]
By \eqref{tg1-2}, we have
\[
\int_{\mathbb{R}^2}[q_k(|x|_b-A)^2+|x|_b(|x|_b-A)]|u_{q_k}(x)|^2\,dx=O((q_k-2)\tau_{q_k}^2).
\]
Observing that
$$
q(|x|_b-A)^2+|x|_b(|x|_b-A)\leq (q+1)|x|_b^2+qA^2
$$ and $\|u_q\|_2^2=1$, we have
\[
\int_{\mathbb{R}^2}|x|_b^2|u_{q_k}(x)|^2\,dx\geq O((q_k-2)\tau_{q_k}^2).
\]
Therefore,
\begin{equation*}
\begin{split}
\int_{\mathbb{R}^2}|x|_b(|x|_b-A)|u_{q_k}(x)|^2\,dx&\geq \int_{\mathbb{R}^2}\big(\frac{1}{2}|x|_b^2-\frac{1}{2}A^2\big)|u_{q_k}(x)|^2\,dx\\
&\geq O((q_k-2)\tau_{q_k}^2).
\end{split}
\end{equation*}
By \eqref{1.1-18} we obtain
\begin{equation*}
\int_{\mathbb{R}^2}|\nabla u_{q_k}(x)|^2\geq \int_{\mathbb{R}^2}|x|_b(|x|_b-A)|u_{q_k}(x)|^2\,dx\geq O((q_k-2)\tau_{q_k}^2),
\end{equation*}
a contradiction.
\end{proof}

\bigskip

\bigskip

{\bf Proof of Theorem 1.2.}  It is known from  \eqref{tg1-2} that
\[
\int_{\mathbb{R}^2}[q(|x|_b-A)^2+|x|_b(|x|_b-A)]|u_q(x)|^2\,dx\leq C(q-2)\tau_q^2.
\]
Then, the fact that
$$
|x|_b(|x|_b-A)\geq \frac{1}{2}(|x|_b^2-A^2)
$$
and $\|u_q\|_2^2=1$ yield
\[
\int_{\mathbb{R}^2}|x|_b^2|u_q(x)|^2\,dx\leq C(q-2)\tau_q^2.
\]
Therefore,
\begin{equation}\label{1.2-1}
\int_{\mathbb{R}^2}(|x|_b-A)^2|u_q(x)|^2\,dx\leq 2\int_{\mathbb{R}^2}(|x|_b^2+A^2)|u_q(x)|^2\,dx\leq C(q-2)\tau_q^2
\end{equation}
and
\begin{equation}\label{1.2-2}
\Big|\int_{\mathbb{R}^2}|x|_b(|x|_b-A)|u_q(x)|^2\,dx\Big|\leq \int_{\mathbb{R}^2}\frac{1}{2}(3|x|_b^2+A^2)|u_q(x)|^2\,dx
\leq C(q-2)\tau_q^2.
\end{equation}
Since $-\Delta+V$ is a compact operator,  its first eigenvalue $\lambda_1$ and corresponding eigenfunction $\varphi_1$ are positive, which satisfy
\[
-\Delta \varphi+V\varphi=\lambda_1\varphi.
\]
Hence, by \eqref{1.1-17},
\begin{equation*}
\begin{split}
\int_{\mathbb{R}^2}\mu_qu_q\varphi_1\,dx
&< \int_{\mathbb{R}^2}(-\Delta u_q+Vu_q)\varphi_1\,dx\\
&=\int_{\mathbb{R}^2}(-\Delta\varphi_1+V\varphi_1)u_q\,dx\\
&=\int_{\mathbb{R}^2}\lambda_1\varphi_1u_q\,dx
\end{split}
\end{equation*}
implying
\begin{equation}\label{1.2-3}
\mu_q< \lambda_1.
\end{equation}
Using \eqref{1.1-17} again, we get
\begin{equation}\label{1.2-4}
\mu_q=\int_{\mathbb{R}^2}|\nabla u_q(x)|^2\,dx
+\int_{\mathbb{R}^2}V(x)|u_q(x)|^2\,dx-a\int_{\mathbb{R}^2}|u_q(x)|^{q+2}\,dx.
\end{equation}
Equation \eqref{1.2-3} and Proposition \ref{tg1} yield for $q>2$ and close to $2$ that
\begin{equation}\label{1.2-5}
a\int_{\mathbb{R}^2}|u_q(x)|^{q+2}\,dx\geq \int_{\mathbb{R}^2}|\nabla u_q(x)|^2\,dx-\lambda_1
\geq \frac{1}{2}\int_{\mathbb{R}^2}|\nabla u_q(x)|^2\,dx.
\end{equation}
By \eqref{1.1-18}, Proposition \ref{tg1} and \eqref{1.2-2},
\begin{equation}\label{1.2-6}
\begin{split}
\frac{qa}{q+2}\int_{\mathbb{R}^2}|u_q(x)|^{q+2}\,dx
&\leq \int_{\mathbb{R}^2}|\nabla u_q(x)|^2\,dx+\Big|\int_{\mathbb{R}^2}|x|_b(|x|_b-A)|u_q(x)|^2\,dx\Big|\\
&\leq C\int_{\mathbb{R}^2}|\nabla u_q(x)|^2\,dx.
\end{split}
\end{equation}

Let $\varepsilon_q=\|\nabla u_q\|_2^{-1}$ and $\tilde{w}_q(x)=\varepsilon_qu_q(\varepsilon_q x)$. Then,
$\|\nabla \tilde{w}_q\|_2^2=\|\tilde{w}_q\|_2^2=1$.
We know from  Lemma \ref{lem:5.1} that $a_q^*\to a^*$ as $q\to 2_+$, and then
\begin{equation*}
\begin{split}
\varepsilon_q^{q-2}
&=\Big(\int_{\mathbb{R}^2}|\nabla u_q(x)|^2\,dx\Big)^{\frac{2-q}{2}}\\
&\geq C^{\frac{2-q}{2}}\tau_q^{2-q}\\
&=C^{\frac{2-q}{2}}\frac{qa}{2a_q^*}\to \frac{a}{a^*}
\end{split}
\end{equation*}
and
\[
\varepsilon_q^{q-2}\leq C^{\frac{2-q}{2}}(q-2)^{\frac{2-q}{2}}\tau_q^{2-q}\to \frac{a}{a^*}.
\]
Thus
\begin{equation}\label{1.2-7}
\varepsilon_q^{q-2}\to \frac{a}{a^*}\ \ {\rm as} \ \ q\to 2_+.
\end{equation}
It follows from \eqref{1.2-5}--\eqref{1.2-7} that there exist $C_1>0$ and $C_2>0$ such that
\begin{equation}\label{1.2-8}
C_1\leq \int_{\mathbb{R}^2}|\tilde{w}_q(x)|^{q+2}\,dx
=\varepsilon_q^{q-2}\frac{\int_{\mathbb{R}^2}|u_q(x)|^{q+2}\,dx}{\int_{\mathbb{R}^2}|\nabla u_q(x)|^2\,dx}
\leq C_2.
\end{equation}

Now we claim that there exist $\{y_q\}\subset \mathbb{R}^2$, $R_0>0$ and $\eta>0$ such that
\begin{equation}\label{1.2-9}
\liminf_{q\to 2_+}\int_{B_{R_0}(y_q)}|\tilde{w}_q(x)|^2\,dx\geq \eta.
\end{equation}
Indeed, if this is not the case, for any $R>0$, there exists a sequence  $\{\tilde{w}_{q_k}\}$ with $q_k\to 2_+$ as $k\to \infty$ such that
\[
\lim_{k\to\infty}\sup_{x\in\mathbb{R}^2}\int_{B_{R}(y)}|\tilde{w}_q(x)|^2\,dx=0.
\]
By Lemma 1.21 in \cite{W1}, we get $\tilde{w}_{q_k}\to 0$ strongly in $L^\gamma (\mathbb{R}^2)$ for any $\gamma>2$, which contradicts to \eqref{1.2-8}.

Denote $w_q(x)=\tilde{w}_q(x+y_q)$. Then
\begin{equation}\label{1.2-10a}
\|\nabla w_q\|_2^2=1;
\end{equation}
\begin{equation}\label{1.2-10b}
C_1\leq \int_{\mathbb{R}^2}|w_q(x)|^{q+2}\,dx\leq C_2;
\end{equation}
\begin{equation}\label{1.2-10c}
\liminf_{q\to 2_+}\int_{B_{R_0}(0)}|w_q(x)|^2\,dx\geq \eta.
\end{equation}
Hence, there exist a sequence $\{q_k\}$, $q_k\to 2_+$ as $k\to \infty$, and $w\in H^1(\mathbb{R}^2)$ such that
$w_{q_k}:=w_k\rightharpoonup w$ weakly in $\mathcal{H}$ and $w_k\to w\neq0$
strongly in $L^\gamma_{loc}(\mathbb{R}^2)$ for any $\gamma>2$.
Denote $\varepsilon_{q_k}:=\varepsilon_k$ and $y_{q_k}=y_k$ for simiplity. By \eqref{1.1-17}, we find that $w_k$ solves
\begin{equation}\label{1.2-10}
-\Delta w_k+\varepsilon_k^2(|\varepsilon_k(x+y_k)|_b-A)^2w_k
=\varepsilon_k^2\mu_{q_k}w_k+\varepsilon_k^{2-q_k}aw_k^{q_k+1}.
\end{equation}

Now, we study the asymptotic behavior of $w_k$. By \eqref{1.2-3}, the definition of $\varepsilon_k$ and Proposition \ref{tg1}, we have
\[
\varepsilon_k^2\mu_{q_k}\leq \varepsilon_k^2\lambda_1\to 0\ \ {\rm as} \ \ k\to \infty.
\]
It follows from \eqref{1.2-4} and \eqref{1.2-5} that
\begin{equation*}
\varepsilon_k^2\mu_{q_k}
\geq -\varepsilon_k^2a\int_{\mathbb{R}^2}|u_{q_k}(x)|^{q_k+2}\,dx
=- \frac{a\int_{\mathbb{R}^2}|u_{q_k}(x)|^{q_k+2}\,dx}{\int_{\mathbb{R}^2}|\nabla u_{q_k}(x)|^{2}\,dx}
\geq -C.
\end{equation*}
So we may assume
\begin{equation}\label{1.2-11}
\varepsilon_k^2\mu_{q_k}\to -\beta^2\leq 0.
\end{equation}

Next, we study the asymptotic behavior of $\varepsilon_k^2(|\varepsilon_k(x+y_k)|_b-A)^2$.

If $\liminf_{k\to\infty}|\varepsilon_ky_k|_b\leq C$, there exists a subsequence of $\{q_k\}$, still denoted by $\{q_k\}$, such that $|\varepsilon_ky_k|_b\leq C$, so for any $M>0$, we have
\begin{equation}\label{1.2-12}
\varepsilon_k^2(|\varepsilon_k(x+y_k)|_b-A)^2\to 0
\end{equation}
uniformly for $x\in B_M(0)$ as $k\to \infty$.

 If $\liminf_{k\to\infty}|\varepsilon_ky_k|_b=\infty$, there exists a subsequence of $\{q_k\}$, still denoted by $\{q_k\}$, such that $|\varepsilon_ky_k|_b\to\infty$, as $k\to \infty$. For any $M>0$, if $|x|_b\leq M$ and $k$ is large enough, we obtain
 \[
 (|\varepsilon_k(x+y_k)|_b-A)^2\geq \frac{1}{2}|\varepsilon_ky_k|_b^2.
 \]
Equation \eqref{1.2-1} then implies that
 \begin{equation*}
\begin{split}
\int_{B_{R_0}(0)}\frac{1}{2}|\varepsilon_ky_k|_b^2|w_k(x)|^2\,dx
&\leq \int_{\mathbb{R}^2}(|\varepsilon_k(x+y_k)|_b-A)^2|w_k(x)|^2\,dx\\
&=\int_{\mathbb{R}^2}(|x|_b-A)^2|u_{q_k}(x)|^2\,dx\\
&\leq C(q_k-2)\tau_{q_k}^2.
\end{split}
\end{equation*}
Consequently, by \eqref{1.2-10a}, we have
\[
|\varepsilon_ky_k|_b^2\leq C(q_k-2)\tau_{q_k}^2.
\]
Let $\varepsilon_q=\|\nabla u_q\|_2^{-1}$. By Proposition \ref{tg1},
\[
\varepsilon_k^2|\varepsilon_ky_k|_b^2\leq \frac{C(q_k-2)\tau_{q_k}^2}{\int_{\mathbb{R}^2}|\nabla u_{q_k}(x)|^2\,dx}\leq C.
\]
We may assume that $\varepsilon_k^2|\varepsilon_ky_k|_b^2\to\mu\geq0$. It is easy to show that for any $M>0$ and $x\in B_M(0)$,
\begin{equation}\label{1.2-13}
\lim_{k\to\infty}\varepsilon_k^2(|\varepsilon_k(x+y_k)|_b-A)^2
=\lim_{k\to\infty}\varepsilon_k^2|\varepsilon_ky_k|_b^2=\mu\geq0.
\end{equation}

In summary of \eqref{1.2-7} and \eqref{1.2-11}--\eqref{1.2-13}, we have the following cases.

Case (i): $\varepsilon_k\mu_{q_k}\to 0$,\ \ and $\varepsilon_k^2(|\varepsilon_k(x+y_k)|_b-A)^2\to0$;

Case (ii): $\varepsilon_k\mu_{q_k}\to -\beta^2<0$,\ \ and $\varepsilon_k^2(|\varepsilon_k(x+y_k)|_b-A)^2\to0$;

Case (iii): $\varepsilon_k\mu_{q_k}\to 0$,\ \ and $\varepsilon_k^2(|\varepsilon_k(x+y_k)|_b-A)^2\to\mu>0$;

Case (iv): $\varepsilon_k\mu_{q_k}\to  -\beta^2<0$,\ \ and $\varepsilon_k^2(|\varepsilon_k(x+y_k)|_b-A)^2\to\mu>0$;

Taking the limit as $k\to \infty$ in \eqref{1.2-10}, we obtain that $w$ satisfies correspondingly in the case (i) that
\begin{equation}\label{1.2-13}
-\Delta w=a^*w^3;
\end{equation}
in the case (ii) that
\begin{equation}\label{1.2-14}
-\Delta w+\beta^2w=a^*w^3;
\end{equation}
in the case (iii) that
\begin{equation*}
-\Delta w+\mu w=a^*w^3;
\end{equation*}
in the case (iv) that
\begin{equation*}
-\Delta w+(\mu+\beta^2)w=a^*w^3.
\end{equation*}

 The case (i) can not happen. Indeed, if it would happen on the contrary, the fact that $w\geq0$ and $w\neq0$ would imply that the equation \eqref{1.2-13} admits a positive solution, which contradicts the Liouville theorem.

 In the case (ii), by the uniqueness of positive solution of \eqref{1.2-14} and $w\neq0$, there exists $y_0\in\mathbb{R}^2$ such that
 \[
 w=\frac{\beta}{\|Q\|_2}Q(\beta(x-y_0)),
 \]
 which implies $\|w\|_2^2=1$. Hence $w_{q_k}\to w$ strongly in $L^2(\mathbb{R}^2)$, that is
 \begin{equation}\label{1.2-15a}
\varepsilon_k u_{q_k}(\varepsilon_k(x+y_k))\to \frac{\beta}{\|Q\|_2}Q(\beta(x-y_0))\ \ {\rm strongly\ \ in} \ \ L^2(\mathbb{R}^2).
\end{equation}
Let $x_k=\varepsilon_k(y_k+y_0)$. Then
\begin{equation}\label{1.2-15}
\bar{w}_k:=\varepsilon_k u_{q_k}(\varepsilon_kx+x_k)\to \frac{\beta}{\|Q\|_2}Q(\beta x) \ \ {\rm strongly \ \ in} \ \ L^2(\mathbb{R}^2).
\end{equation}
The case (iii) and (iv) can be treated in the same way as the case (ii), we omit the detail. The proof is complete.
\qquad $\Box$

\bigskip

Finally, we study the limiting behavior of $x_k$.

{\bf Proof of Theorem \ref{thm4}}
It is obvious that  either $\liminf_{k\to \infty}|x_k|=+\infty$ or $\liminf_{k\to \infty}|x_k|$ is bounded.

If $\liminf_{k\to \infty}|x_k|=+\infty$, there is nothing to prove.

If $\liminf_{k\to \infty}|x_k|$ is bounded, there exist $x_0\in \mathbb{R}^2$ and a subsequence of $\{x_k\}$, still denoted by $\{x_k\}$ such that $x_k\to x_0$, as $k\to \infty$. We may show as the proof of Theorem \ref{thm3} that only the case (ii) may happen, that is,
$\varepsilon_k^2\mu_{q_k}\to -\beta^2<0$ and $\varepsilon_k^{2-q_k}a\to a^*$. By \eqref{1.2-10}, we have
\[
-\Delta w_{q_k}\leq \varepsilon_k^{2-q_k}a w_k^{q_k+1}.
\]
Using the De Diorgi-Nash-Moser estimate, see Theorem 4.1 in \cite{HL}, we obtain
\begin{equation*}
\begin{split}
\max_{x\in B_1(\xi)}w_k(x)
&\leq C\Big(\int_{B_2(\xi)}|w_k(x)|^2\,dx\Big)^{\frac{1}{2}}\\
&\leq C\Big[(\int_{B_2(\xi)}|w_k(x)-w|^2\,dx\Big)^{\frac{1}{2}}+\Big(\int_{B_2(\xi)}|w(x)|^2\,dx\Big)^{\frac{1}{2}}\Big].
\end{split}
\end{equation*}
Since $w_{q_k}\to w$ strongly in $L^2(\mathbb{R}^2)$, $w_k$ is uniformly bounded in $L^\infty(\mathbb{R}^2)$. Hence, there exists $R>0$ such that if $k$ is large enough and $|x|\geq R$,
\[
\varepsilon_k^{2-q_k}aw_k^{q_k}\leq \frac{\beta^2}{2}
\]
 and
 \[
 -\Delta w_k\leq -\frac{1}{3}\beta^2w_k.
 \]
 By the comparison principle,
 \[
  w_k\leq Ce^{-\frac{1}{4}\beta^2|x|}\ \ {\rm for} \ \ |x| \geq R.
 \]
 Since $w_k$ is uniformly bounded in $L^\infty(\mathbb{R}^2)$, we have
 \[
  w_k\leq Ce^{-\frac{1}{4}\beta^2|x|}\ \ {\rm for \ \ any \ \ }x\in \mathbb{R}^2.
 \]
Noting $x_k=\varepsilon_k(y_k+y_0)$ is bounded in $\mathbb{R}^2$, we have
 \begin{equation}\label{1.2-16}
\int_{\mathbb{R}^2}(|x|_b-A)^2|u_{q_k}(x)|^2\,dx=\int_{\mathbb{R}^2}(|\varepsilon_k(x+y_k)|_b-A)^2|w_k(x)|^2\,dx\leq C.
\end{equation}
Similarly,
\begin{equation}\label{1.2-17}
\Big|\int_{\mathbb{R}^2}|x|_b(|x|_b-A)|u_{q_k}(x)|^2\,dx\Big|\leq C\ \ \ {\rm and} \ \ \int_{\mathbb{R}^2}|x|_b^2|u_{q_k}(x)|^2\,dx\leq C.
\end{equation}
By \eqref{tg1-2}, \eqref{1.2-16} and \eqref{1.2-17}, we obtain
\begin{equation}\label{1.2-19}
\lim_{k\to\infty}\tau_k^{-2}\int_{\mathbb{R}^2}|\nabla u_{q_k}|^2\,dx=1.
\end{equation}
Let
\begin{equation*}
\begin{split}
g(t)
&=:E_{a,q_k}(u_{q_k}^t)\\
&=\frac{1}{2}t^2\int_{\mathbb{R}^2}|\nabla u_{q_k}|^2\,dx-\frac{at^{q_k}}{q_k+2}\int_{\mathbb{R}^2}|u_{q_k}(x)|^{q_k+2}\,dx\\
&+\frac{1}{2}t^{-2}\int_{\mathbb{R}^2}|x|_b^2|u_{q_k}(x)|^2\,dx-At^{-1}\int_{\mathbb{R}^2}|x|_b|u_{q_k}(x)|^2\,dx
+\frac{1}{2}A^2\int_{\mathbb{R}^2}|u_{q_k}(x)|^2\,dx\\
&:=\frac{1}{2}Bt^2-\frac{1}{q}Ct^{q_k}+\frac{1}{2}Dt^{-2}-Et^{-1}+F.
\end{split}
\end{equation*}
Then
\begin{equation*}
\begin{split}
g'(t)&=Bt-Ct^{q_k-1}-Dt^{-3}+Et^{-2}\\
&=t^{-3}(Bt^{4}-Ct^{q_k+2}-D+Et)\\
&:=t^{-3}h(t),
\end{split}
\end{equation*}
\begin{equation*}
h'(t)=4Bt^3-C(q_k+2)t^{q_k+1}+E,
\end{equation*}
and
\begin{equation*}
h''(t)=12Bt^2-C(q_k+2)(q_k+1)t^{q_k}.
\end{equation*}
It is easy to deduce that there exist $t_1>0$ and $t_2>0$ such that \\
$(i)$ $g(t)$ is decreasing in $(0,t_1)$ and $(t_2,+\infty)$, while $g(t)$ is increasing in $(t_1,t_2)$;\\
$(ii)$ $g'(t_1)=g'(t_2)=0$ and $\lim_{t\to0}g(t)=\infty$.

In order to estimate $t_1$ and $t_2$, we rewrite $g'(t)$ as
\begin{equation*}
\begin{split}
g'(t)
&=t\int_{\mathbb{R}^2}|\nabla u_{q_k}|^2\,dx-\frac{q_ka}{q_k+2}t^{q_k-1}\int_{\mathbb{R}^2}|u_{q_k}(x)|^{q_k+2}\,dx\\
&-\frac{1}{t^3}\int_{\mathbb{R}^2}|x|_b^2|u_{q_k}(x)|^2\,dx+At^{-2}\int_{\mathbb{R}^2}|x|_b|u_{q_k}(x)|^2\,dx.
\end{split}
\end{equation*}
By the Pohozaev identity \eqref{1.1-18}, we have $g'(1)=0$.

Now, we claim that $t_2=1$. Indeed, we choose $t_3=C(q_k-2)^{\frac{1}{2}}$ such that
\[
\frac{1}{2}\int_{\mathbb{R}^2}\Big|\nabla u_{q_k}^{t_3}(x)|^2\,dx\leq \frac{q_k-2}{4q_k}\tau_k^2.
\]
Then by \eqref{1.2-17}, we have
\begin{equation*}
\begin{split}
\frac{1}{2}\int_{\mathbb{R}^2}V(x)|u_{q_k}^{t_3}(x)|^2\,dx
&=\frac{1}{2}\int_{\mathbb{R}^2}(t_3^{-1}|x|_b-A)^2|u_{q_k}(x)|^2\,dx\\
&\leq \int_{\mathbb{R}^2}(t_3^{-2}|x|_b^2+A^2)|u_{q_k}(x)|^2\,dx\\
&\leq C(q_k-2)^{-2}.
\end{split}
\end{equation*}
By Proposition \ref{tg1}, for $q>2$ and $q$ close to $2$ we have,
\begin{equation*}
\begin{split}
g(t_3)&\leq \frac{1}{2}\Big(\int_{\mathbb{R}^2}\Big|\nabla u_{q_k}^{t_3}(x)\Big|^2\,dx
+\int_{\mathbb{R}^2}V(x)\Big| u_{q_k}^{t_3}(x)\Big|^2\,dx\Big)\\
&\leq \frac{q_k-2}{4q_k}\tau_k^2+C(q_k-2)^{-2}\\
&\leq \frac{q_k-2}{2q_k}\tau_k^2+o(1)\\
&=E_{a,q_k}(u_{q_k})=g(1).
\end{split}
\end{equation*}
Hence,  $t_2=1$ and
\begin{equation}\label{1.2-20}
g(1)=E_{a,q_k}(u_{q_k})\geq g(t)=E_{a,q_k}(u_{q_k}^t)\ \ {\rm for \ \ any} \ \ t\in (C(q_k-2)^{\frac{1}{2}}, +\infty).
\end{equation}
Choose $t_k$ such that
\[
\tilde{Q}_{q_k}(u_{q_k}^{t_k})=t_k^2\int_{\mathbb{R}^2}|\nabla u_{q_k}(x)|^2\,dx-\frac{q_ka}{q_k+2}t_k^{q_k}
\int_{\mathbb{R}^2}| u_{q_k}(x)|^{q_k+2}\,dx=0.
\]
By \eqref{1.1-18}, \eqref{1.2-17} and \eqref{1.2-19}, we obtain
\begin{equation*}
\begin{split}
t_k&=\left[\frac{(q_k+2)\int_{\mathbb{R}^2}|\nabla u_{q_k}(x)|^2\,dx}{q_ka\int_{\mathbb{R}^2}| u_{q_k}(x)|^{q_k+2}\,dx}\right]^{\frac{1}{q_k-2}}\\
&=\left[\frac{\int_{\mathbb{R}^2}|\nabla u_{q_k}(x)|^2\,dx}{\int_{\mathbb{R}^2}|\nabla u_{q_k}(x)|^2\,dx-\int_{\mathbb{R}^2}|x|_b(|x|_b-A)| u_{q_k}(x)|^2\,dx}\right]^{\frac{1}{q_k-2}}\\
&=\big[1+O(\tau_k^{-2})\big]^{\frac{1}{q_k-2}}.
\end{split}
\end{equation*}
Therefore,
\begin{equation}\label{1.2-22}
\big|t_k-1\big|=\Big|e^{\frac{1}{q_k-2}\log (1+O(\tau_k^{-2}))}-1\Big|=\frac{1}{q_k-2}O(\tau_k^{-2})\ \ {\rm as} \ \ k\to\infty
\end{equation}
and $t_k>t_3$ for $k$ large enough.

 By Lemma \ref{lem:3.1} and Remark \ref{r1}, we get
\begin{equation}\label{1.2-21}
\inf_{g\in \tilde{\Gamma}_q}\max_{t\in[0,1]}E_{a,q}|_{V=0}(g(t))=\inf\{E_{a,q}|_{V=0}(u):\tilde{Q}_q(u)=0, u\in \tilde{S}(1) \}=\frac{q-2}{2q}\tau_q^2.
\end{equation}
We derive from \eqref{1.2-20} and \eqref{1.2-21} that
\begin{equation*}
\begin{split}
g(1)&=E_{a,q_k}(u_{q_k})\\
&\geq g(t_k)\\
&=E_{a,q_k}(u_{q_k}^{t_k})\\
&=E_{a,q_k}|_{V=0}(u_{q_k}^{t_k})+\int_{\mathbb{R}^2}V(x)\big|u_{q_k}^{t_k}(x)\big|^2\,dx\\
&\geq \inf\{E_{a,q}|_{V=0}(u):\tilde{Q}_q(u)=0, u\in \tilde{S}(1) \}+\int_{\mathbb{R}^2}V(x)\big|u_{q_k}^{t_k}(x)\big|^2\,dx\\
&=\frac{q_k-2}{2q_k}\tau_k^2+\int_{\mathbb{R}^2}V(x)\big|u_{q_k}^{t_k}(x)\big|^2\,dx.
\end{split}
\end{equation*}
Then, Proposition \ref{ee} yields
\begin{equation}\label{1.2-22b}
\limsup_{k\to\infty}\tau_k^2\int_{\mathbb{R}^2}V(x)\big|u_{q_k}^{t_k}(x)\big|^2\,dx\leq
\frac{1}{2\|Q\|_2^2}\int_{\mathbb{R}^2}|x|^2|Q(x)|^2\,dx.
\end{equation}
Next we claim that
\begin{equation}\label{1.2-23}
\liminf_{k\to\infty}\Big|\frac{|x_k|_b-At_k}{\varepsilon_kt_k}\Big|\leq C.
\end{equation}
We argue by contradiction. Suppose this is not the case,  there would exist a subsequence of $\{q_k\}$, still denoted by $\{q_k\}$, such that \[
\frac{||x_k|_b-At_{k}|}{\varepsilon_kt_k}\to \infty.
\]
By the definition $\varepsilon_k$, \eqref{1.2-10c} and \eqref{1.2-19}, for any $C\geq 0$ and $k$ large enough,
\begin{equation*}
\begin{split}
\tau_k^2\int_{\mathbb{R}^2}V(x)\big|u_{q_k}^{t_k}(x)\big|^2\,dx
&\geq \frac{1}{2}\varepsilon_k^{-2}\int_{\mathbb{R}^2}V(x)\big|u_{q_k}^{t_k}(x)\big|^2\,dx\\
&=\frac{1}{2}\varepsilon_k^{-2}\int_{\mathbb{R}^2}\Big(\frac{|\varepsilon_ky+x_k|}{t_k}-A\Big)^2|\bar{w}_k(y)|^2\,dy\\
&\geq \frac{1}{2}\int_{B_{R_0}(y_0)}\Big(\big|\frac{y}{t_k}+\frac{x_k}{\varepsilon_kt_k}\big|-\frac{A}{\varepsilon_k}\Big)^2
|\bar{w}_k(y)|^2\,dy\\
&\geq C,
\end{split}
\end{equation*}
which contradicts \eqref{1.2-22b}.
Hence, we may assume that
\begin{equation}\label{4.57a}
\frac{|x_k|_b-At_k}{\varepsilon_kt_k}\to C_0\ \ {\rm and} \ \ x_k\to x_0\ \ {\rm with} \ \ |x_0|_b=A.
\end{equation}
By the Fatou Lemma, \eqref{1.2-15} and \eqref{1.2-19}, we have
\begin{equation*}
\begin{split}
&\liminf_{k\to \infty}\tau_k^2\int_{\mathbb{R}^2}V(x)\big|u_{q_k}^{t_k}(x)\big|^2\,dx\\
&=\liminf_{k\to \infty}\varepsilon_k^{-2}\int_{\mathbb{R}^2}V(x)\big|u_{q_k}^{t_k}(x)\big|^2\,dx\\
&=\liminf_{k\to \infty}
\int_{\mathbb{R}^2}\Big[\Big(\big|\frac{y}{t_k}+\frac{x_k}{\varepsilon_kt_k}\big|_b
-\frac{|x_k|_b}{\varepsilon_kt_k}\Big)+\Big(\frac{|x_k|_b}{\varepsilon_kt_k}-\frac{A}{\varepsilon_k}\Big)\Big]^2
|\bar{w}_k(y)|^2\,dy\\
&=\liminf_{k\to \infty}\int_{\mathbb{R}^2}
\Big[\frac{1}{\varepsilon_kt_k(|\varepsilon_ky+x_k|_b+|x_k|_b)}
\big[\frac{\varepsilon_k^2|y_1|^2}{b_1^2}+\frac{\varepsilon_k^2|y_2|^2}{b_2^2}
+\frac{2\varepsilon_kx_k^1y_1}{b_1^2}+\frac{2\varepsilon_kx_k^2y_2}{b_2^2}
\big]\\
&~~+\frac{|x_k|_b-At_k}{\varepsilon_kt_k}\Big]^2|\bar{w}_k(y)|^2\,dy\\
&\geq \int_{\mathbb{R}^2}\Big(\frac{1}{|A|}\big(\frac{x_0^1y_1}{b_1^2}+\frac{x_0^2y_2}{b_2^2}\big)+C_0\Big)^2\frac{|Q(y)|^2}{\|Q\|_2^2}\,dy.
\end{split}
\end{equation*}
Since $Q$ is symmetric, we deduce  that
\begin{equation*}
\begin{split}
&\liminf_{k\to\infty}\tau_k^2\int_{\mathbb{R}^2}V(x)\big|u_{q_k}^{t_k}(x)\big|^2\,dx\\
&\geq \frac{1}{A^2\|Q\|_2^2}\int_{\mathbb{R}^2}\big(\frac{|x_0^1|^2|y_1|^2}{b_1^4}
+\frac{|x_0^2|^2|y_2|^2}{b_2^4}\big)|Q(y)|^2\,dy+C_0^2\\
&=\frac{1}{2A^2\|Q\|_2^2}\big(\frac{|x_0^1|^2}{b_1^4}+\frac{|x_0^2|^2}{b_2^4}\big)
\int_{\mathbb{R}^2}|y|^2|Q(y)|^2\,dy+C_0^2\\
&=\frac{1}{2A^2\|Q\|_2^2}\Big[\frac{1}{b_1^2}\big(\frac{|x_0^1|^2}{b_1^2}+\frac{|x_0^2|^2}{b_2^2}
\big)+\frac{|x_0^2|^2}{b_2^2}\big(\frac{1}{b_2^2}-\frac{1}{b_1^2}\big)\Big]
\int_{\mathbb{R}^2}|y|^2|Q(y)|^2\,dy+C_0^2\\
&=\frac{1}{2A^2\|Q\|_2^2}\Big[\frac{1}{b_1^2}A^2+\frac{|x_0^2|^2}{b_2^2}\big(\frac{1}{b_2^2}-\frac{1}{b_1^2}\big)\Big]
\int_{\mathbb{R}^2}|y|^2|Q(y)|^2\,dy+C_0^2\\
&=\frac{1}{2A^2\|Q\|_2^2}\int_{\mathbb{R}^2}|y|^2|Q(y)|^2\,dy
+\frac{|x_0^2|^2}{2A^2\|Q\|_2^2b_2^2}\big(\frac{1}{b_2^2}-\frac{1}{b_1^2}\big)
\int_{\mathbb{R}^2}|y|^2|Q(y)|^2\,dy+C_0^2.
\end{split}
\end{equation*}
By \eqref{1.2-22b}, this implies $x_0^2=0$ and $C_0=0$, and by \eqref{4.57a},  $x_0\in \mathcal{Z}$ and
\begin{equation}\label{1.2-24}
\frac{|x_k|_b}{\varepsilon_kt_k}-\frac{A}{\varepsilon_k}\to0.
\end{equation}
Finally, we further study the limiting behavior of $x_k$. Noting $\varepsilon_q=\|\nabla u_q\|_2^{-1}$, by \eqref{1.2-19}, \eqref{1.2-22}, \eqref{1.2-24} and $x_k\to x_0$, we get
\begin{equation*}
\begin{split}
\lim_{k\to\infty}\tau_k||x_k|_b-A|
&=\lim_{k\to\infty}\frac{||x_k|_b-A|}{\varepsilon_k}\\
&\leq\lim_{k\to\infty}\frac{|x_k|_b}{\varepsilon_kt_k} |t_k-1|+\lim_{k\to\infty}\frac{1}{\varepsilon_kt_k}||x_k|_b-At_k|\\
&\leq\lim_{k\to\infty}\frac{C|x_k|_b}{\varepsilon_kt_k}\frac{\tau_k^{-2}}{q_k-2}+\lim_{k\to\infty}\frac{1}{\varepsilon_kt_k}||x_k|_b-At_k|\\
&=\lim_{k\to\infty}\frac{C\varepsilon_k|x_k|_b}{t_k(q_k-2)}+\lim_{k\to\infty}\frac{1}{\varepsilon_kt_k}||x_k|_b-At_k|\\
&\to 0.
\end{split}
\end{equation*}
The proof is complete.\qquad $\Box$

\bigskip
\vspace{2mm}
\noindent{\bf Acknowledgment} The first author is supported by NNSF of China, No:11671179 and 11771300. The second author is supported by NNSF of China, No:11701260.

{\small
\end{document}